\documentclass[11pt]{article}
\usepackage[T1]{fontenc}
\usepackage[latin9]{inputenc}
\usepackage{geometry}
\usepackage{amsmath}
\usepackage{amsthm}
\usepackage{amssymb}
\usepackage{graphicx}
\usepackage[unicode=true,pdfusetitle,
 bookmarks=true,bookmarksnumbered=false,bookmarksopen=false,
 breaklinks=false,pdfborder={0 0 1},backref=false,colorlinks=false]
 {hyperref}

\makeatletter

\numberwithin{equation}{section}
\theoremstyle{plain}
\newtheorem{thm}{\protect\theoremname}[section]
  \theoremstyle{plain}
  \newtheorem{cor}[thm]{\protect\corollaryname}
  \theoremstyle{remark}
  \newtheorem{rem}[thm]{\protect\remarkname}
  \theoremstyle{remark}
  \newtheorem{claim}[thm]{\protect\claimname}
  \theoremstyle{definition}
  \newtheorem{defn}[thm]{\protect\definitionname}
  \theoremstyle{plain}
  \newtheorem{lem}[thm]{\protect\lemmaname}
  \theoremstyle{plain}
  \newtheorem{prop}[thm]{\protect\propositionname}

\date{}
\usepackage{mathrsfs}

\let\stdpart\part
\renewcommand*{\part}{\clearpage\stdpart}

\usepackage{cite}
\usepackage{arcs}

\usepackage[format=hang,labelfont=bf]{caption}

\@ifundefined{showcaptionsetup}{}{%
 \PassOptionsToPackage{caption=false}{subfig}}
\usepackage{subfig}
\makeatother

  \providecommand{\claimname}{Claim}
  \providecommand{\corollaryname}{Corollary}
  \providecommand{\definitionname}{Definition}
  \providecommand{\lemmaname}{Lemma}
  \providecommand{\propositionname}{Proposition}
  \providecommand{\remarkname}{Remark}
\providecommand{\theoremname}{Theorem}

\begin{document}
\global\long\def\H{\mathbb{H}}
\global\long\def\RR{\mathbb{R}}
\global\long\def\C{\mathbb{C}}
\global\long\def\Z{\mathbb{Z}}
\global\long\def\Q{\mathbb{Q}}
\global\long\def\K{\mathbb{K}}
\global\long\def\ZZ{\mathbb{Z}}

\global\long\def\gam{\Gamma}
\global\long\def\ga{\gamma}
\global\long\def\ph{\varphi}
\global\long\def\ang{\Upsilon}
\global\long\def\lam{\Lambda}
\global\long\def\lm{\lambda}
\global\long\def\ee{\varepsilon}

\global\long\def\om{\omega}
\global\long\def\z{\zeta}
\global\long\def\e{\epsilon}
\global\long\def\a{\alpha}
\global\long\def\b{\beta}
\global\long\def\dl{\delta}
\global\long\def\t{\theta}
\global\long\def\s{\sigma}

\global\long\def\mn{\mu_{N}}
\global\long\def\sl{\mbox{SL}_{2}}
\global\long\def\psl{\mbox{PSL}_{2}}
\global\long\def\gmod{G/\Gamma}
\global\long\def\half{\frac{1}{2}}
\global\long\def\im{\Im}
\global\long\def\re{\Re}
\global\long\def\vol{\mbox{Vol}}
\global\long\def\span{\mbox{span}}

\global\long\def\B{{\cal B}}
\global\long\def\Ocal{{\cal O}}

\global\long\def\hyp{\mathbf{H}}

\global\long\def\nv{\left\Vert v\right\Vert }
\global\long\def\del{\partial}
\global\long\def\l{\ell}
\global\long\def\ra{\rightarrow}

\global\long\def\rmax{r_{\Psi}}
\global\long\def\jmax{j_{\max}}

\global\long\def\rectlow#1{R_{#1}}
\global\long\def\rectup#1{R^{#1}}

\global\long\def\ball{B}
\global\long\def\point{z}

\title{Horospherical coordinates of lattice points in hyperbolic space:
effective counting and equidistribution}

\author{Tal Horesh\thanks{Department of Mathmatics, Technion, Israel, \texttt{hotal@tx.technion.ac.il}.}
\and Amos Nevo\thanks{Department of Mathmatics, Technion, Israel, \texttt{anevo@tx.technion.ac.il}. Supported by ISF Grant No.~2095/15.}}
\maketitle
\begin{abstract}
We establish effective counting and equidistribution results for lattice
points in families of domains in hyperbolic spaces, of any dimension
and over any field. The domains we focus on are defined as product
sets with respect to the Iwasawa decomposition. Several classical
Diophantine problems can be reduced to counting lattice points in
such domains, including distribution of shortest solution to the $\gcd$
equation, and angular distribution of primitive vectors in the plane.
We give an explicit and effective solution to these problems, and
extend them to imaginary quadratic number fields. Further applications
include counting lifts of closed horospheres to hyperbolic manifolds
and establishing an equidistribution property of integral solutions
to the Diophantine equation defined by a Lorentz form.
\end{abstract}
\tableofcontents{}

\section{Introduction and statement of main results}

A lattice in a Lie group is a discrete subgroup whose fundamental
domain has finite Haar measure. Our goal in the present paper is to
establish effective counting and equidistribution results for Iwasawa
components of lattice elements in real rank one Lie groups that are
simple up to a finite center; namely, isometry groups of hyperbolic
spaces. These problems are instances of hyperbolic counting problems,
in which one seeks to study the asymptotic behavior of the number
of lattice orbit points in some expanding family of regions in hyperbolic
space, and generalize the classical question of counting in hyperbolic
balls.

A natural extension of the much studied class of counting problems
in the euclidean space, hyperbolic counting problems typically have
the property that the number of lattice points inside sufficiently
regular domains is asymptotic to the volume of these domains. In both
settings there is a special interest in estimating the error term,
i.e. the difference between the volume of a domain and the number
of lattice points inside it. Unlike the euclidean setting, in which
one can produce a bound for the error term in terms of the volume
of a thin neighborhood of the boundary of suitable domains, the hyperbolic
setting presents a special challenge; this is due to the fact that
a fixed proportion of the volume, and of the lattice points, is concentrated
near the boundary.

Counting the points of a lattice orbit in a hyperbolic space can be
easily deduced from counting the elements of the lattice subgroup
itself inside the group of isometries $G$; the approach we take is
the one of counting in the actual group.

The domains that we consider are product sets in the Iwasawa coordinates
on $G$: $G=NAK$, where $K$ is maximal compact, $A\cong\RR$, and
$N$ is the unipotent subgroup that stabilizes the ideal boundary
point $\left\{ \infty\right\} $. The map $N\times A\times K\ra G$
given by $\left(n,a,k\right)\mapsto nak$ is a diffeomorphism, so
these are indeed coordinates on $G$. For example, $\mbox{SL}_{2}\left(\RR\right)$
decomposes into
\begin{equation}
\begin{array}{cclcc}
N & = & \left\{ \left[\begin{array}{cc}
1 & x\\
0 & 1
\end{array}\right]:x\in\RR\right\} \\
A & = & \left\{ \left[\begin{array}{cc}
e^{t/2} & 0\\
0 & e^{-t/2}
\end{array}\right]:t\in\RR\right\} \\
K & = & \left\{ \left[\begin{array}{cc}
\cos\t & -\sin\t\\
\sin\t & \cos\t
\end{array}\right]:0\leq\t<2\pi\right\}  & =\mbox{SO}\left(2\right)
\end{array}\label{eq: NAK SL(2,R)}
\end{equation}

Let $G$ denote a non-exceptional simple Lie group of real rank one
with finite center; namely, locally isomorphic to one of the following:
$\mbox{SO}\left(1,n\right)$, $\mbox{SU}\left(1,n\right)$, or $\mbox{SP}\left(1,n\right)$
for some $n\geq1$. The corresponding rank $1$ symmetric spaces $G/K$
are, respectively: the real hyperbolic space $\hyp_{\RR}^{n}$, the
complex hyperbolic space $\hyp_{\C}^{n}$ and the quaternionic hyperbolic
space $\hyp_{\H}^{n}$. Every $G$ acts on the corresponding space
by isometries of the Riemannian distance, which we will refer to as
the ``hyperbolic distance'' and denote by $d\left(\cdot,\cdot\right)$.
The remaining rank one simple Lie group is $\mbox{F}_{4\left(-20\right)}$,
which corresponds to the octonionic hyperbolic plane $\hyp_{\mathbb{O}}^{2}$;
we shall not consider this case.

A Haar measure $\mu$ on $G$ is given in the Iwasawa coordinates
as follows. As in the above example, $A$ is parametrized such that
$A=\left\{ a_{t}:t\in\RR\right\} $, and $d\left(a_{t}\cdot i,a_{s}\cdot i\right)=\left|t-s\right|$,
where $i$ is the point that $K$ stabilizes in the symmetric space.
Let $\mu_{K}$ denote a Haar measure on $K$. The subgroup $N$ is
parametrized by a euclidean space of the appropriate dimension (see
table \ref{tab: rank 1 lie groups} for the different cases), and
a Haar measure on $N$ is the Lebesgue measure on this underlying
euclidean space. A Haar measure on $G$ w.r.t. the Iwasawa coordinates
is given by
\begin{equation}
\mu=\mn\times\frac{dt}{e^{2\rho t}}\times\mu_{K},\label{eq: Haar measure on G}
\end{equation}
where $\rho$ is a parameter that depends on the group $G$. The Iwasawa
subgroups, symmetric spaces and Haar measure of the rank one groups
are summarized in table \ref{tab: rank 1 lie groups}.

We will mainly focus on lattices $\gam<G$ that are non-cocompact;
without loss of generality we may assume that such $\gam$ has a cusp
at $\infty$. We consider lattice points whose $N$ and $K$ components
lie in given bounded subsets $\Psi\subset N$ and $\Phi\subseteq K$,
and study their asymptotic behavior as their $A$-components tend
to $\infty$. When the lattice has a cusp at $\infty$, there are
only finitely many lattice points in $\Psi A\Phi$ whose $A$-coordinate
is positive. This finite number of points clearly does not affect
the asymptotics, and it is therefore sufficient to consider lattice
points in the family $\left\{ \rectlow T\left(\Psi,\Phi\right)\right\} _{T>0}$,
where
\begin{eqnarray*}
\rectlow T\left(\Psi,\Phi\right) & := & \Psi A_{\left[-T,0\right]}\Phi=\left\{ na_{t}k:n\in\Psi,t\in\left[-T,0\right],k\in\Phi\right\}
\end{eqnarray*}
(Figure \ref{fig: Strip S_T}) as $T\ra\infty$. According to \ref{eq: Haar measure on G},
the volume of these domains equals
\[
\mu\left(\rectlow T\left(\Psi,\Phi\right)\right)=\frac{1}{2\rho}\cdot\mn\left(\Psi\right)\mu_{K}\left(\Phi\right)\left(e^{2\rho T}-1\right).
\]

\begin{figure}
\hspace*{-5mm}
\subfloat[{\label{fig: strip in upper plane}}]%
{\includegraphics[scale=0.45]
{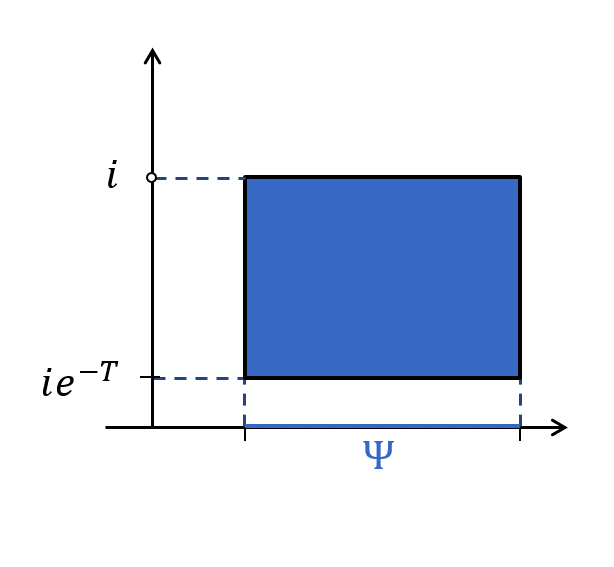}}
\hspace*{\fill}
\subfloat[\label{fig: strip in upper space}]%
{\includegraphics[scale=0.45]
{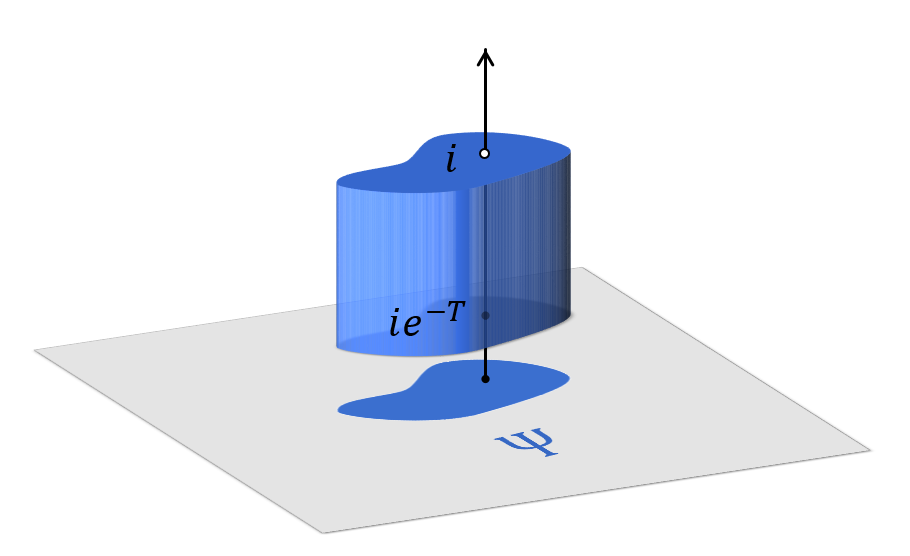}}
\hspace*{\fill}
\subfloat[\label{fig: strip in disc}]%
{\includegraphics[scale=0.45]
{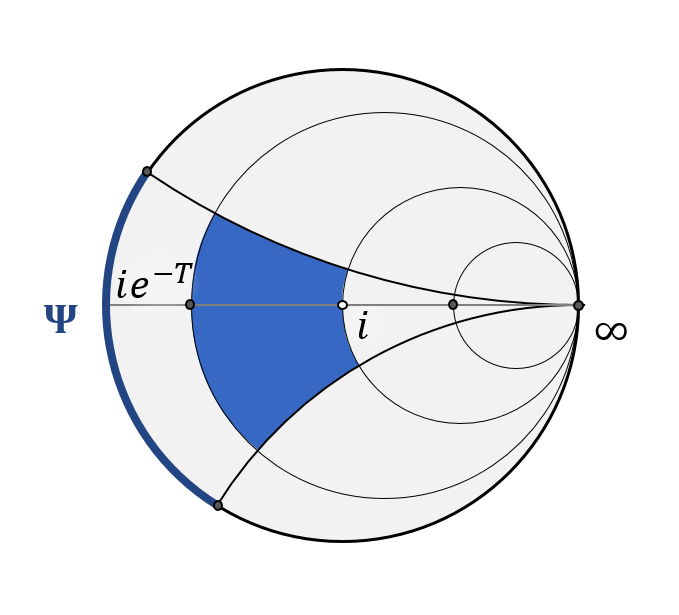}}
\hspace*{\fill}

\caption{The domains $\protect\rectlow T\left(\Psi,\Phi\right)$ projected
to: (a) real hyperbolic plane in upper half plane model, (b) real
hyperbolic $3$-space in upper half space model, (c) real hyperbolic
plane in unit disc model. \label{fig: Strip S_T}}
\end{figure}
We shall require that the domains $\Psi\subset N$
and $\Phi\subseteq K$ are \emph{nice}: bounded, full dimension embedded
submanifolds whose boundaries are piecewise smooth \textemdash{} namely,
a finite union of submanifolds of co-dimension $1$. We allow the
case where only some of these submanifolds are included in the nice
set, while others are not, e.g. a half open rectangle that two of
its edges are included and the remaining two are not included.
\begin{thm}
\label{thm: Counting in rectangles}Let $\Psi\subset N$ and $\Phi\subseteq K$
be nice domains, and consider the family $\rectlow T\left(\Psi,\Phi\right)$
as defined above. For any lattice $\gam<G$, there exists a parameter
$\kappa=\kappa\left(\gam\right)<1$ (defined explicitly in \ref{eq: kappa lattice exponent})
such that for $T>0$:
\begin{eqnarray*}
\#\left(\rectlow T\left(\Psi,\Phi\right)\cap\gam\right) & = & \frac{\mu\left(\rectlow T\left(\Psi,\Phi\right)\right)}{\mu\left(G/\gam\right)}+O\left(\log\left(\mu\left(\rectlow T\left(\Psi,\Phi\right)\right)\right)\cdot\mu\left(\rectlow T\left(\Psi,\Phi\right)\right)^{\kappa}\right)\\
 & = & \frac{\mn\left(\Psi\right)\mu_{K}\left(\Phi\right)}{\mu\left(G/\gam\right)}\cdot\frac{e^{2\rho T}}{2\rho}+O\left(T\left(e^{2\rho T}\right)^{\kappa}\right).
\end{eqnarray*}
The implicit constant depends on $\Psi$ and $\Phi$.
\end{thm}
For example, for the lattice $\mbox{SL}_{2}\left(\Z\right)$ in $\mbox{SL}_{2}\left(\RR\right)$
$\kappa\left(\mbox{SL}_{2}\left(\Z\right)\right)=7/8$, and Theorem
\ref{thm: Counting in rectangles} produces the best known error estimate
for this case. This particular case has received considerable attention,
which we will briefly detail at the end of the next section. We note
that while (as noted above) the domains $\rectlow T\left(\Psi,\Phi\right)$
are natural to consider in the context of lattices with a cusp, Theorem
\ref{thm: Counting in rectangles} applies for any lattice $\gam<G$.
However, when the lattice in question is co-compact, the cuspidal
strip $\Psi A_{\left(0,\infty\right)}\Phi$ may contain infinitely
many lattice points, despite its bounded volume. The irregularity
caused by this cuspidal strip is the reason why the domains $\rectlow T$
must be truncated at height $t=0$. In order to study the co-compact
case, one should consider the sets $\Psi A_{\left[-T,T\right]}\Phi$,
which we will do elsewhere.

For $H\in\left\{ N,A,K\right\} $, we denote the projection to the
$H$-component by $\pi_{H}:G\ra H$.
\begin{cor}
\label{cor: Equidistribution of N and K}Let $\Psi,\Psi'\subset N$
and $\Phi,\Phi'\subseteq K$ be nice, and let $\gam<G$ be any lattice.
For $0<T$,
\[
\frac{\#\left(\gam\cap\rectlow T\left(\Psi',\Phi'\right)\right)}{\#\left(\gam\cap\rectlow T\left(\Psi,\Phi\right)\right)}=\frac{\mn\left(\Psi'\right)\mu_{K}\left(\Phi'\right)}{\mn\left(\Psi\right)\mu_{K}\left(\Phi\right)}+O\left(T\left(e^{2\rho T}\right)^{-\left(1-\kappa\right)}\right)
\]
where the implied constant depends on $\Psi,\Psi',\Phi,\Phi'$ and
$\kappa=\kappa\left(\gam\right)<1$ is the exponent associated with
$\gam$ appearing in Theorem \ref{thm: Counting in rectangles}.

\begin{enumerate}
\item The set of $N$-components $\left\{ \pi_{N}\left(\ga\right):\ga\in\gam\cap\Psi A_{\left[-T,0\right]}\Phi\right\} $
become effectively equidistributed in $\Psi$ w.r.t. $\mu_{N}$ as
$T\ra\infty$. Namely, for every compactly supported Lipschitz function
$f$ on $N$,
\[
\left|\frac{1}{\#\left(\gam\cap\Psi A_{\left[-T,0\right]}\Phi\right)}\cdot\sum_{\ga\in\Psi A_{\left[-T,0\right]}\Phi}f\left(\pi_{N}\left(\ga\right)\right)-\frac{1}{\mn\left(\Psi\right)}\cdot\int_{\Psi}f\,d\mn\right|\leq\mbox{const}\cdot Te^{-2\rho T\left(1-\kappa\right)},
\]
where the constant depends on the function $f$.\label{enu: N e.d.}
\item The set of $K$-components $\left\{ \pi_{K}\left(\ga\right):\ga\in\gam\cap\Psi A_{\left[-T,0\right]}\Phi\right\} $
become effectively equidistributed in $\Phi$ w.r.t. $\mu_{K}$ as
$T\ra\infty$ (implying an analogous statement to the one in \ref{enu: N e.d.}
for a compactly supported Lipschitz function on $K$).\label{enu: K e.d.}
\end{enumerate}
\end{cor}
The proofs for Theorem \ref{thm: Counting in rectangles} and Corollary \ref{cor: Equidistribution of N and K}
are in Section \ref{sec: Proof of Thm Rectangles}.
\begin{rem}
\label{rem: KAN vs NAK}Iwasawa decomposition of a Lie group is used
in one of two conventions: $G=NAK$ or $G=KAN$. Our results are phrased
with respect to the first option, but the corresponding statements
with respect to the $KAN$ decomposition may be easily deduced. Indeed,
the $KAN$ coordinates of $g\in G$ are obtained from the $NAK$ coordinates
of $g^{-1}$: $g^{-1}=nak$ implies $g=k^{-1}a^{-1}n^{-1}$. In particular,
the Haar measure with respect to the $KAN$ coordinates is $\mu_{K}\times e^{2\rho t}dt\times\mn$,
and the statement of Theorem \ref{thm: Counting in rectangles} is
replaced by
\[
\#\gam\cap\left(\Phi A_{\left[0,T\right]}\Psi\right)=\frac{\mn\left(\Psi\right)\mu_{K}\left(\Phi\right)}{\mu\left(G/\gam\right)}\cdot\frac{e^{2\rho T}}{2\rho}+O\left(T\left(e^{2\rho T}\right)^{\kappa}\right).
\]
for $\Phi\subset K$, $\Psi\subset N$ and $\kappa$ as in Theorem
\ref{thm: Counting in rectangles}, and $T>0$.
\end{rem}

\begin{rem}
Note that Theorem \ref{thm: Counting in rectangles} was formulated
for a family of domains in $G$ itself, rather than in the symmetric
space; this enables us to analyze the distribution of the $K$-components
of the lattice elements. As we shall see below, equidistribution of
the $K$-components plays a key role in a number of applications,
including angular equidistribution of shortest solutions to the $\gcd$
equation in $\Z^{2}$. The connection between the problem of equidistribution
of the norms of the shortest solutions and the equidistribution of
Iwasawa N-components in $\mbox{SL}_{2}\left(\Z\right)$ was first
pointed out by Risager and Rudnick \cite{R&R}, and has motivated
the approach pursued in the present paper. We will first formulate
and prove our results and then comment further on the history of this
problem.
\end{rem}

\section{Iwasawa components and diophantine problems\label{sec: Number-theoretical applications}}

\subsection{Distribution of shortest solutions of the gcd equation}

We now turn to some consequences of Corollary \ref{cor: Equidistribution of N and K}
for certain integral lattices in a real hyperbolic space of small
dimension. In what follows, the norm we refer to is the euclidean
norm on $\RR^{2}$ or $\C^{2}$, denoted by $\left\Vert \cdot\right\Vert $.

For every primitive integral vector $v=\left(a,b\right)$, let $w_{v}$
denote the shortest integral vector that completes $v$ to a (positively
oriented) basis of $\Z^{2}$, namely, the shortest solution to the
$\gcd$ equation $bx-ay=1$. Let $\t_{v}$ denote the angle from $w_{v}$
to $v$ (anticlockwise). We say that $v$ is \emph{positive} if $\t_{v}$
is acute, and \emph{negative} if $\t_{v}$ is obtuse (Figure \ref{fig: R-R ext for integers - v and w_v}).

In the case of the lattice $\mbox{SL}_{2}\left(\Z\right)$, Corollary
\ref{cor: Equidistribution of N and K} has the following geometric
interpretation.

\begin{figure}
\begin{centering}
\includegraphics[scale=0.5]{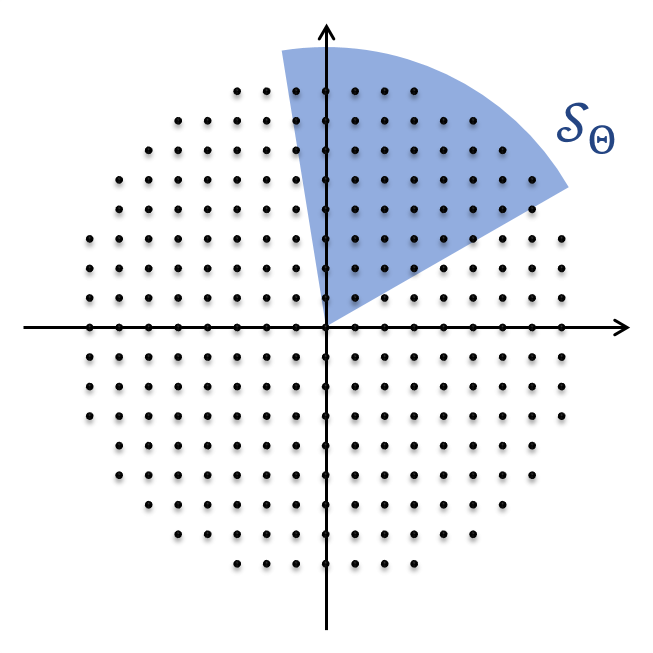}
\par\end{centering}
\caption{\label{fig: lattice points in sector}$\protect\Z^{2}$-points contained
in the sector $\mathcal{S}_{\Theta}$.}
\end{figure}

\begin{thm}
\label{thm: R-R extended in integers} Let $\Theta\subseteq S^{1}$
an arc in the unit circle, and let $\mathcal{S}_{\Theta}$ be the
corresponding sector of the plane $\RR^{2}$ (see Figure \ref{fig: lattice points in sector}).
For every primitive integral vector $v=\left(a,b\right)$, let $w_{v}$
and $\t_{v}$ as above. For $v\in\mathcal{S}_{\Theta}$, $\left\Vert v\right\Vert \ra\infty$:

\begin{enumerate}
\item The ratios $\left\Vert w_{v}\right\Vert /\left\Vert v\right\Vert $
of the length of the shortest solution relative to the length of $v$
become effectively equidistributed in $\left[0,1/2\right]$;\label{enu: e.d. of quotients of norms}
\item The values $v/\left\Vert v\right\Vert $ become effectively equidistributed
in $\Theta$;\label{enu: angular distribution of primitve vectors}
\item The values $w_{v}/\left\Vert w_{v}\right\Vert $ become effectively
equidistributed in $\Theta$ when $v$ is restricted to positive vectors,
and in $-\Theta$ when $v$ is restricted to negative vectors. In
particular, when $\left|\Theta\right|\leq\pi$, the values $w_{v}/\left\Vert w_{v}\right\Vert $
become effectively equidistributed in $\Theta$, in $-\Theta$ and
in $\Theta\cup-\Theta$;\label{enu: angular distribution of shortest solutions}
\item Parts \ref{enu: e.d. of quotients of norms} and \ref{enu: angular distribution of primitve vectors}
hold when $v$ is restricted to positive vectors only, or to negative
vectors only.\label{enu: strengthening for positive/negative vectors}
\end{enumerate}
In all the above effective equidistribution statements, the rate of
convergence is $O\left(\nv^{-1/4}\cdot\log\nv\right)$.
\end{thm}
\begin{rem}
Note that part \ref{enu: angular distribution of primitve vectors}
asserts that the directions of primitive integral vectors in any sector
in the plane converge to uniform distribution on the corresponding
arc of the circle at a rate given by the radius to the power of $-1/4$.
This result may well be known, but we have not been able to find it
in the literature.
\end{rem}

\begin{proof}[Proof of Theorem \ref{thm: R-R extended in integers}]
If $v=\left(a,b\right)\in\Z^{2}$ is primitive, it can be completed
to countably many matrices in $\sl\left(\Z\right)$, representing
the different integral solutions to the equation $bx-ay=1$; The $NAK$
components of these integral matrices encode the vector $v$ and the
different solutions to $bx-ay=1$ as follows. Recall that the Iwasawa
decomposition of $\sl\left(\RR\right)$ given in \ref{eq: NAK SL(2,R)};
if $\left(x,y\right)$ is such a solution, the corresponding matrix
in $\sl\left(\Z\right)$ has $NAK$ decomposition
\[
\left[\begin{array}{cc}
x & y\\
a & b
\end{array}\right]=\underset{N\mbox{-component}}{\underbrace{\left(\begin{array}{cc}
1 & \frac{xa+yb}{a^{2}+b^{2}}\\
 & 1
\end{array}\right)}}\underset{A\mbox{-component}}{\underbrace{\left(\begin{array}{cc}
\frac{1}{\sqrt{a^{2}+b^{2}}}\\
 & \sqrt{a^{2}+b^{2}}
\end{array}\right)}}\underset{K\mbox{-component}}{\underbrace{\frac{1}{\sqrt{a^{2}+b^{2}}}\left(\begin{array}{cc}
b & -a\\
a & b
\end{array}\right)}}.
\]
Note that the $A$ and $K$ components depend only on the vector $v$:
the $A$-component $\left[\begin{smallmatrix}1/\left\Vert v\right\Vert  & 0\\
0 & \left\Vert v\right\Vert
\end{smallmatrix}\right]$ encodes the norm of $v$, and the $K$-component $\left[\begin{smallmatrix} & v^{\perp}/\left\Vert v\right\Vert \\
 & v/\left\Vert v\right\Vert
\end{smallmatrix}\right]$ encodes the angle of $v=\left(a,b\right)$ w.r.t the positive real
axis. The $N$-component depends on the specific solution $\left(x,y\right)$,
namely the upper row of the matrix; if $w:=\left(x,y\right)$, then
the $N$-component is $\left[\begin{smallmatrix}1 & \left\langle w,v\right\rangle /\left\Vert v\right\Vert ^{2}\\
0 & 1
\end{smallmatrix}\right]$ (the projection of $w$ to the line $\span\left\{ v\right\} $, divided
by the norm of $v$).

Throughout the rest of the proof we shall identify the subgroups $N$,
$A$, and $K$ of $\sl\left(\RR\right)$ given in \ref{eq: NAK SL(2,R)}
with $\RR$, $\RR$, and $S^{1}$ respectively through $\left[\begin{smallmatrix}1 & x\\
0 & 1
\end{smallmatrix}\right]\leftrightarrow x$, $\left[\begin{smallmatrix}e^{t/2} & 0\\
0 & e^{-t/2}
\end{smallmatrix}\right]\leftrightarrow t$, and $\left[\begin{smallmatrix}\cos\t & -\sin\t\\
\sin\t & \cos\t
\end{smallmatrix}\right]\leftrightarrow\t$.

The different solutions to $bx-ay=1$ are $\left\{ \left(x+ma,y+mb\right):m\in\Z\right\} $,
and they correspond to matrices $\left[\begin{smallmatrix}x+ma & y+mb\\
a & b
\end{smallmatrix}\right]$ whose $N$-coordinates are
\[
\frac{\left(x+ma\right)a+\left(y+mb\right)b}{a^{2}+b^{2}}=m+\frac{xa+yb}{a^{2}+b^{2}}=m+\frac{\left\langle w,v\right\rangle }{\left\Vert v\right\Vert ^{2}}
\]
(namely, all the integral translations of the real number $\left\langle w,v\right\rangle /\left\Vert v\right\Vert ^{2}$).
Observe that, among all the integral matrices that correspond to $v$,
the one whose $N$-coordinate is minimal \textemdash{} i.e., in the
interval $\left[-1/2,1/2\right)$ \textemdash{} is the one that corresponds
to the shortest solution to $bx-ay=1$, namely, the one whose upper
row has minimal norm. This is because the integral solutions are the
integral points on the affine line $\span\left\{ v\right\} +w$ (where
$w$ is any solution) which is parallel to $\span\left\{ v\right\} $;
hence when decomposing $\RR^{2}$ as $\span\left\{ v\right\} \oplus\span\left\{ v^{\perp}\right\} $,
all of these solutions have the same $v^{\perp}$ component, and the
shortest integral solution is the one with the shortest $v$-component
(namely, the shortest projection on $\span\left\{ v\right\} $). The
shortest integral solution $w_{v}$ corresponds to the matrix $\left[\begin{smallmatrix}w_{v}\\
v
\end{smallmatrix}\right]=\left[\begin{smallmatrix}x_{v} & y_{v}\\
a & b
\end{smallmatrix}\right]$, which we denote by $\ga_{v}$.

We conclude that the set $\left\{ \ga\in\sl\left(\Z\right):\pi_{N}\left(\ga\right)\in\left[-1/2,1/2\right)\right\} $
is in one-to-one correspondence $\ga_{v}\leftrightarrow v$ with the
set of primitive integral vectors, where for each $\ga_{v}$, $\pi_{K}\left(\ga_{v}\right)$
is the angle of the corresponding primitive vector $v$ and $e^{-\pi_{A}\left(\ga_{v}\right)/2}$
is the length of $v$. Thus, applying part \ref{enu: K e.d.} of Corollary
\ref{cor: Equidistribution of N and K} with $\gam=\sl\left(\Z\right)$,
$e^{T/2}\geq\left\Vert v\right\Vert $, and $\Phi=\Theta$ in $S^{1}$
(these will remain fixed throughout the proof), as well as $\Psi=\left[-1/2,1/2\right)$,
proves statement \ref{enu: angular distribution of primitve vectors}
of the theorem; indeed, as we shall see in Section \ref{subsec: The GN method},
$\kappa\left(\sl\left(\Z\right)\right)=7/8$. We also remark that
since this is an equidistribution argument, it does not matter if
we replace the half-closed interval $\left[-1/2,1/2\right)$ by the
closed interval $\left[-1/2,1/2\right]$.

Recall that $\t_{v}$ denotes the angle from $w_{v}$ to $v$ (anticlockwise),
hence the $N$-component of $\ga_{v}$ is given by
\begin{equation}
\pi_{N}\left(\ga_{v}\right)=\frac{x_{v}a+y_{v}b}{a^{2}+b^{2}}=\frac{\left\langle w_{v},v\right\rangle }{\left\Vert v\right\Vert ^{2}}=\frac{\left\Vert w_{v}\right\Vert \cos\left(\t_{v}\right)}{\left\Vert v\right\Vert }.\label{eq: N-coordinate of gamma_v}
\end{equation}
Since
\[
1=\det\left(\left[\begin{array}{cc}
x_{v} & y_{v}\\
a & b
\end{array}\right]\right)=\det\left(\left[\begin{array}{c}
w_{v}\\
v
\end{array}\right]\right)=\left\Vert w_{v}\right\Vert \left\Vert v\right\Vert \left|\sin\left(\t_{v}\right)\right|,
\]
and $\left\Vert w_{v}\right\Vert \geq1$, it follows that $\left|\sin\left(\t_{v}\right)\right|=O\left(\left\Vert v\right\Vert ^{-1}\right)$,
or equivalently $1-\left|\cos\left(\t_{v}\right)\right|=O\left(\left\Vert v\right\Vert ^{-2}\right)$.
From part \ref{enu: N e.d.} of Corollary \ref{cor: Equidistribution of N and K}
applied to $\Psi=\left[-1/2,1/2\right]$, we have that for primitive
vectors $v$ in ${\cal S}_{\Theta}$, the $N$-components of $\ga_{v}$
become uniformly equidistributed in $\left[-1/2,1/2\right]$ as $\left\Vert v\right\Vert \ra\infty$,
at rate $O\left(\nv^{-1/4}\cdot\log\nv\right)$. Clearly, this means
that their absolute values become uniformly equidistributed in $\left[0,1/2\right]$
at the same rate. These absolute values are
\[
\left|\pi_{N}\left(\ga_{v}\right)\right|=\frac{\left\Vert w_{v}\right\Vert }{\left\Vert v\right\Vert }\cdot\left|\cos\left(\t_{v}\right)\right|=\frac{\left\Vert w_{v}\right\Vert }{\left\Vert v\right\Vert }\left(1+O\left(\left\Vert v\right\Vert ^{-2}\right)\right)=\frac{\left\Vert w_{v}\right\Vert }{\left\Vert v\right\Vert }+O\left(\left\Vert v\right\Vert ^{-2}\right);
\]
thus, the values $\left\Vert w_{v}\right\Vert /\left\Vert v\right\Vert $
are also uniformly equidistributed in $\left[0,1/2\right]$ at rate
$O\left(\nv^{-1/4}\cdot\log\nv\right)$ (since $\left\Vert v\right\Vert ^{-2}<\nv^{-1/4}\cdot\log\nv$),
which proves part \ref{enu: e.d. of quotients of norms}.

By the computation \ref{eq: N-coordinate of gamma_v} of the $N$-component
of $\ga_{v}$, the vector $v$ is positive if and only if $\cos\t_{v}\geq0$,
i.e. if and only if $\pi_{N}\left(\ga_{v}\right)\geq0$. Alternatively,
$v$ is negative if and only if $\pi_{N}\left(\ga_{v}\right)\leq0$.
Thus, by applying Corollary \ref{cor: Equidistribution of N and K}
to $\Psi=\left[0,1/2\right]$, we obtain part \ref{enu: strengthening for positive/negative vectors}
for the positive vectors, and similarly when $\Psi=\left[-1/2,0\right]$,
we obtain part \ref{enu: strengthening for positive/negative vectors}
for the negative vectors.

We now restrict attention to positive vectors $v$, where $\t_{v}$
is acute and therefore $\t_{v}\approx\sin\t_{v}=O\left(\left\Vert v\right\Vert ^{-1}\right)$.
In particular, the direction of the vector $w_{v}$ approaches the
direction of $v$ at rate $O\left(\left\Vert v\right\Vert ^{-1}\right)$.
Since the directions of the positive primitive vectors $v$ are uniformly
equidistributed in $\Theta$ with rate $O\left(\nv^{-1/4}\cdot\log\nv\right)$
(by part \ref{enu: strengthening for positive/negative vectors}),
then so do the directions of the vectors $w_{v}$.

Consider the negative vectors $v$, where $\t_{v}$ is obtuse and
therefore $\pi-\t_{v}\approx\sin\t_{v}=O\left(\left\Vert v\right\Vert ^{-1}\right)$.
Now the direction of the vector $w_{v}$ approaches the direction
of $-v$ at rate $O\left(\left\Vert v\right\Vert ^{-1}\right)$. Since
the directions of the negative primitive vectors $v$ are uniformly
equidistributed in $\Theta$ with rate $O\left(\nv^{-1/4}\cdot\log\nv\right)$,
the directions of their negatives $-v$ become uniformly equidistributed
in $-\Theta$ at the same rate; it follows that the directions of
the vectors $w_{v}$ become uniformly equidistributed in $-\Theta$
with rate $O\left(\nv^{-1/4}\cdot\log\nv\right)$, which concludes
the proof of part \ref{enu: angular distribution of shortest solutions}.
\end{proof}
\global\long\def\detv{\det\left(\begin{smallmatrix}w\\
 v
\end{smallmatrix}\right)}
\global\long\def\line{W}

\begin{figure}
\hspace*{\fill}\subfloat[$\protect\t_{v}$ acute \textemdash{} $v$ positive]{\includegraphics[scale=0.65]{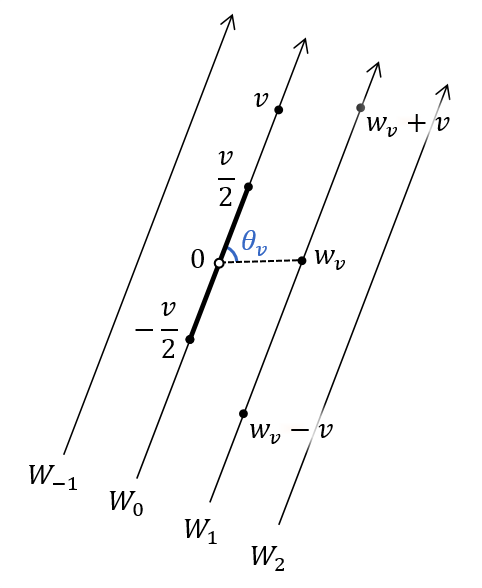}%
}\hspace*{\fill}\subfloat[$\protect\t_{v}$ obtuse \textemdash{} $v$ negative]{%
\includegraphics[scale=0.65]{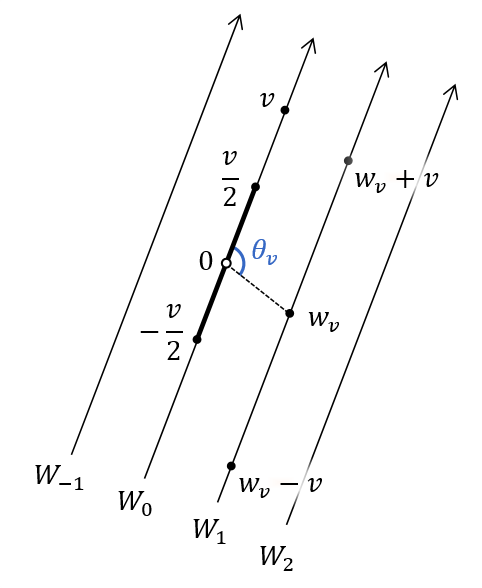}%
}\hspace*{\fill}

\caption{$v$, $w_{v}$ and $\protect\t_{v}$. This Figure also depicts the
lines $\protect\line_{m}=\left\{ w:\protect\detv=m\right\} $ for
$m\in\mathbb{Z}$, where $\protect\line_{0}=\protect\span\left\{ v\right\} $
and $w_{v}$ is the shortest integral vector in $\protect\line_{1}$.
\label{fig: R-R ext for integers - v and w_v}. }
\end{figure}
Theorem \ref{thm: R-R extended in integers}
extends to rings of integers in imaginary quadratic number fields
as follows. Let $d$ be a positive square free integer, and let ${\cal O}_{d}$
denote the ring of integers in the quadratic number field $\Q\left[\sqrt{-d}\right]$.
The ring ${\cal O}_{d}$ is a lattice in $\C$, and has a fundamental
parallelogram
\[
\mathcal{P}_{d}=\left\{ z\,:\,-\half\leq\re\left(z\right)<\half,\,-\frac{\sqrt{\left|\mbox{Disc}\left(d\right)\right|}}{4}\leq\im\left(z\right)<\frac{\sqrt{\left|\mbox{Disc}\left(d\right)\right|}}{4}\right\}
\]
(e.g. \cite{EGMbook}), where $\mbox{Disc}\left(d\right)$ is the
discriminant of $\Q\left[\sqrt{-d}\right]$. The rectangle $\mathcal{P}_{d}$
is symmetric w.r.t. the origin, hence all its vertices have the same
norm, which we denote by $\rho_{d}$. We let $\nu_{d}$ denote the
probability measure on $\left[0,\rho_{d}\right]$ which is the distribution
of the norm of a random point in $\mathcal{P}_{d}$. Note that $\nu_{d}$
is not Lebesgue measure. We refer to $v=\left(\a,\b\right)$ in $\mathcal{O}_{d}^{2}$
as \emph{primitive} if the ideals $\left\langle \a\right\rangle $
and $\left\langle \b\right\rangle $ are co-prime; namely, if there
exists a solution $\left(\xi,\eta\right)$ in $\mathcal{O}_{d}^{2}$
to $\a\xi-\b\eta=1$.
\begin{thm}
\label{thm: R-R extends to number fields}Let $\Theta\subseteq S^{3}$
be a spherical cap in the unit sphere, and let $\mathcal{S}_{\Theta}$
be the corresponding sector of $\RR^{4}\cong\C^{2}$ . For every primitive
vector $v=\left(\a,\b\right)\in\mathcal{O}_{d}^{2}$, let $w_{v}$
denote the shortest vector that completes $v$ to a basis of $\mathcal{O}_{d}^{2}$,
namely, the shortest ${\cal O}_{d}$-integral solution to the equation
$\a\xi-\b\eta=1$. For $v\in\mathcal{S}_{\Theta}$, $\left\Vert v\right\Vert \ra\infty$:

\begin{enumerate}
\item The values $v/\left\Vert v\right\Vert $ become effectively equidistributed
in $\Theta$;\label{enu: O_d angular distribution of primitve vectors}
\item The ratios $\left\Vert w_{v}\right\Vert /\left\Vert v\right\Vert $
of the length of the shortest ${\cal O}_{d}$-integral solution relative
to the length of $v$ become effectively equidistributed in $\left[0,\rho_{d}\right]$
with respect to $\nu_{d}$. \label{enu:  O_d  e.d. of quotients of norms }
\end{enumerate}
In all the above effective equidistribution statements, the rate of
convergence is $O\left(\left\Vert v\right\Vert ^{4\left(1-\kappa_{d}\right)}\cdot\log\nv\right)$,
where $\kappa_{d}$ is the exponent that corresponds to the lattice
$\mbox{PSL}_{2}\left(\mathcal{O}_{d}\right)$ of $\mbox{PSL}_{2}\left(\C\right)$
in Theorem \ref{thm: Counting in rectangles}.
\end{thm}
Observe that when $\mathcal{O}_{d}$ is a euclidean domain, i.e. when
$d\in\left\{ 1,2,3,7,11\right\} $, $\a$ and $\b$ are co-prime and
the equation $\a\xi-\b\eta=1$ is their $\gcd$ equation. Thus, $w_{v}$
is the shortest solution to the $\gcd$ equation defined by $v$,
as in the case of $\Z$ which was discussed in Theorem \ref{thm: R-R extended in integers}.

All the arguments in the proof of Theorem \ref{thm: R-R extended in integers}
carry through to the proof of Theorem \ref{thm: R-R extends to number fields},
where this time Corollary \ref{cor: Equidistribution of N and K}
is applied for the Iwasawa components of the lattice $\mbox{SL}_{2}\left({\cal O}_{d}\right)$
in $\mbox{SL}_{2}\left(\C\right)$. We briefly describe the necessary
adjustments.
\begin{proof}
Recall the Iwasawa decomposition of $\mbox{SL}_{2}\left(\C\right)$
consists of the subgroups:
\[
\begin{array}{cclcc}
N & = & \left\{ \left[\begin{array}{cc}
1 & z\\
 & 1
\end{array}\right]:z\in\C\right\} \\
A & = & \left\{ \left[\begin{array}{cc}
e^{t/2} & 0\\
0 & e^{-t/2}
\end{array}\right]:t\in\RR\right\} \\
K & = & \left\{ \left[\begin{array}{cc}
\overline{b} & -\overline{a}\\
a & b
\end{array}\right]:\left|a\right|^{2}+\left|b\right|^{2}=1\right\}  & =\mbox{SU}\left(2\right).
\end{array}
\]
Clearly, $K$ is isomorphic to the unit sphere $S^{3}$ in $\C^{2}$.

A primitive pair $\left(\a,\b\right)\in{\cal O}_{d}^{2}$ can be completed
to a matrix $\left[\begin{smallmatrix}\xi & \eta\\
\a & \b
\end{smallmatrix}\right]$ in $\mbox{SL}_{2}\left({\cal O}_{d}\right)$, and the Iwasawa coordinates
of such a matrix are
\[
\left[\begin{array}{cc}
\xi & \eta\\
\a & \b
\end{array}\right]=\underset{N\mbox{-component}}{\underbrace{\left(\begin{array}{cc}
1 & \frac{\xi\overline{\a}+\eta\overline{\b}}{\left\Vert \a\right\Vert ^{2}+\left\Vert \b\right\Vert ^{2}}\\
 & 1
\end{array}\right)}}\underset{A\mbox{-component}}{\underbrace{\left(\begin{array}{cc}
\frac{1}{\sqrt{\left\Vert \a\right\Vert ^{2}+\left\Vert \b\right\Vert ^{2}}}\\
 & \sqrt{\left\Vert \a\right\Vert ^{2}+\left\Vert \b\right\Vert ^{2}}
\end{array}\right)}}\underset{K\mbox{-component}}{\underbrace{\frac{1}{\sqrt{\left\Vert \a\right\Vert ^{2}+\left\Vert \b\right\Vert ^{2}}}\left(\begin{array}{cc}
\overline{\b} & -\overline{\a}\\
\a & \b
\end{array}\right)}}.
\]
The $A$ and $K$ components encode the vector $v$: the $A$-component
encodes its norm, and the $K$-component encodes its projection to
the sphere $S^{3}$. The $N$-component corresponds to the upper row:
if $w=\left(\xi,\eta\right)$, this component equals $\left\langle w,v\right\rangle /\left\Vert v\right\Vert ^{2}$.
The set of solutions to $\b\xi-\a\eta=1$ is $\left\{ \left(\xi+m\b,\eta+m\a\right):m\in\mathcal{O}_{d}\right\} $,
and the matrices in $\mbox{SL}_{2}\left({\cal O}_{d}\right)$ that
correspond to these solutions differ only by their $N$-components;
these components are
\[
\left[\begin{array}{cc}
1 & \frac{\left(\xi+m\a\right)\overline{\a}+\left(\eta+m\b\right)\overline{\b}}{\left\Vert \a\right\Vert ^{2}+\left\Vert \b\right\Vert ^{2}}\\
 & 1
\end{array}\right]=\left[\begin{array}{cc}
1 & m+\frac{\xi\overline{\a}+\eta\overline{\b}}{\left\Vert \a\right\Vert ^{2}+\left\Vert \b\right\Vert ^{2}}\\
 & 1
\end{array}\right]=\left[\begin{array}{cc}
1 & m+\frac{\left\langle w,v\right\rangle }{\left\Vert v\right\Vert ^{2}}\\
 & 1
\end{array}\right],
\]
where $m\in\mathcal{O}_{d}$. By the same Pythagorean argument that
was used in the real case (Theorem \ref{thm: R-R extended in integers}),
the shortest ${\cal O}_{d}$-integral solution $w_{v}=\left(\xi_{v},\eta_{v}\right)$
to $\b\xi-\a\eta=\detv=1$ corresponds to the matrix $\ga_{v}\in\mbox{SL}_{2}\left({\cal O}_{d}\right)$
whose $N$-coordinate is minimal.

Clearly, $\left\{ m+\frac{\left\langle w,v\right\rangle }{\left\Vert v\right\Vert ^{2}}:m\in\mathcal{O}_{d}\right\} $
is a coset of the lattice ${\cal O}_{d}$ in $\C$, hence there is
a unique element from this coset in every ${\cal O}_{d}$-integral
translation of the fundamental domain $\mathcal{P}_{d}$. The representative
which is of minimal norm is the one that lies in $\mathcal{P}_{d}$
itself. Thus, $\ga_{v}$ is the unique matrix in $\mbox{SL}_{2}\left({\cal O}_{d}\right)$,
among the matrices that correspond to $v$, whose $N$-component lies
in $\mathcal{P}_{d}$.

Let $\mbox{s}\left(v\right)$ and $\mbox{c}\left(v\right)$ be such
that
\[
1=\det\left(\left[\begin{array}{cc}
\xi_{v} & \eta_{v}\\
\a & \b
\end{array}\right]\right)=\det\left(\left[\begin{array}{c}
w_{v}\\
v
\end{array}\right]\right)=\left\Vert w_{v}\right\Vert \left\Vert v\right\Vert \cdot\mbox{s}\left(v\right),
\]
and
\[
\left\langle w_{v},v\right\rangle =\left\Vert w_{v}\right\Vert \left\Vert v\right\Vert \cdot\mbox{c}\left(v\right)
\]
(these are the analogs for $\sin\left(\t_{v}\right)$ and $\cos\left(\t_{v}\right)$
from the proof of Theorem \ref{thm: R-R extends to number fields}).
It can be verified that
\[
\left|\mbox{s}\left(v\right)\right|^{2}+\left|\mbox{c}\left(v\right)\right|^{2}=1,
\]
and in particular, when $\mbox{s}\left(v\right)$ is small, $\left|\mbox{s}\left(v\right)\right|^{2}=1-\left|\mbox{c}\left(v\right)\right|^{2}\approx1-\left|\mbox{c}\left(v\right)\right|$.
Since $\mbox{s}\left(v\right)=O\left(\left\Vert v\right\Vert ^{-1}\right)$,
we have $1-\left|\mbox{c}\left(v\right)\right|=O\left(\left\Vert v\right\Vert ^{-2}\right)$,
and now the proof proceeds analogously to the one of Theorem \ref{thm: R-R extends to number fields},
by applying Corollary \ref{cor: Equidistribution of N and K} to $\gam=\mbox{SL}_{2}\left({\cal O}_{d}\right)$,
$e^{T/2}\geq\left\Vert v\right\Vert $, $\Psi=\mathcal{P}_{d}$ and
$\Phi=\Theta$.
\end{proof}

\subsection{Counting and equidistribution of Iwasawa coordinates: history of
the problem\label{subsec: History}}

\paragraph*{From the $\gcd$ equation in $\protect\Z^{2}$ to equidistribution
of real parts of lattice orbits.}

The problem of analyzing the distribution of the shortest solution
to the gcd equation in $\Z^{2}$ was considered by Dinaburg and Sinai
\cite{Sinai} who measured the size of the shortest solution by the
maximum norm, and used the theory of continued fractions. It was subsequently
noted by Risager and Rudnick \cite{R&R} that when the size of the
smallest solution is measured using the Euclidean norm, equidistribution
of shortest solutions is equivalent to the problem of equidistribution
of real parts of the points in the orbit of $i$ under $\mbox{SL}_{2}\left(\Z\right)$
in the upper half plane, and the latter result has already been established
by Good \cite{Good}. Truelsen \cite{Truelsen} has established, using
estimates of exponential sums, a quantitative form for the equidistribution
of real parts for any lattice with a standard cusp in $\mbox{SL}_{2}\left(\RR\right)$.
In particular this establishes a rate of convergence in the equidistribution
of shortest solutions of the gcd equation, and this rate is slightly
improved upon in Theorem \ref{thm: R-R extended in integers}.

\paragraph*{Counting above fixed intervals in the upper half plane.}

Truelsen's result is based on establishing the existence of the right
number of points in the orbit of $\Gamma$ inside $N_{I}A_{\left[-T,0\right]}\cdot i$
for any interval $I$ contained in $\left[-\half,\half\right]$ up
to an error of lower order. We note that establishing equidistribution
depends on solving the lattice point counting problem associated with
\emph{every} interval $I$. In the specific case of $I=\left[-\half,\half\right]$,
better error terms were established by Good for general non-cocompact
lattices in $\mbox{SL}_{2}\left(\RR\right)$, and by Chamizo \cite{Chamizo}
for the specific lattice orbit $\mbox{SL}_{2}\left(\Z\right)\cdot i$.
Observe that this particular case is equivalent to the \emph{primitive
circle problem}. Chamizo, and later on Truelsen, have established
some further lattice point counting results for a variety of other
families in the upper half plane.

\paragraph*{Lifts of horospheres in hyperbolic space.}

Eskin and McMullen \cite{EM93} have raised the problem of counting
the number of lifts of a closed horosphere ${\cal H}$ in $G/\gam$
which intersect a ball of radius $T$ in hyperbolic space, and have
established the main term for this counting problem. As we shall see
below, in the case of the hyperbolic space, the problem amounts to
counting the points of a non-cocompact lattice lying in the sets $\rectlow T\left(\Psi,K\right)$,
where $\Psi\subset\RR^{n-1}$ is the fundamental domain of $\gam\cap N$.
The problem can be formulated for an arbitrary symmetric space, and
the main term of the asymptotics has been established by \cite{Mohammadi_Golsefidy}.
Further work on the subject has been recently carried out in \cite{DKL_16}
and \cite{Shi_15}.

\paragraph*{Local statistics of the Iwasawa $N$-component.}

Marklof and Vinogradov \cite{Marklof_Vinogradov} have considered,
among other things, the projection of lattice orbit points to a neighborhood
of a horizontal horosphere tending to the boundary, namely the sets
given by $\rectlow T\left(\Psi,K\right)\setminus\rectlow{T-c}\left(\Psi,K\right)$.
They have analyzed the local statistics of the Iwasawa $N$-components
in $\Psi$ as $T\ra\infty$; this problem is more delicate than just
the equidistribution of the $N$-component, and was established in
real hyperbolic space of any dimension, but not in an effective form.

\paragraph*{Contribution of the current paper.}

We note that the counting and equidistribution results in the present
paper are effective, namely include an error estimate. In the case
of the lattice $\sl\left(\ZZ\right)$, for which quantitative results
have been established in \cite{Truelsen}, our error estimate reduces
from the previous $O\left(e^{7T/8+\e}\right)$ for every $\e>0$,
to $O\left(Te^{7T/8}\right)$. The method that we utilize uses only
the size of the spectral gap in the automorphic representation, and
avoids the detailed spectral analysis of eigenfunctions of the Laplacian,
and any use of Kloosterman sums which appeared in previous arguments.
This fact is what allows an easy generalization to any dimension and
any group of real rank one, and it is also responsible for elimination
of the $\e$ in the error exponent, reducing the main spectral estimate
to known estimates of the Harish-Chandra $\Xi$ function. Our consideration
of lattice points and equidistribution of their Iwasawa components
in the group itself, rather than a lattice orbit in the symmetric
space, allows us to obtain equidistribution of the $K$-components
of the elements of a lattice, in addition to their $N$-components.
This fact, along with our consideration of dimensions greater than
$2$, enables us to extend the results of Risager and Rudnick to include
angular equidistribution, and rings of integers in $\C$.

\section{Extensions and applications\label{sec: Applications}}

Recall that $2\rho$ is the exponent that appears in the Haar measure
\ref{eq: Haar measure on G} of $G$ when given in Iwasawa coordinates.

\subsection{Difference of domains of the form $\protect\rectlow T$}

Theorem \ref{thm: Counting in rectangles} can be extended to include
the case where the horosphere that delimits the domains $\rectlow T\left(\Psi,\Phi\right)$
from above is not fixed at height $1$. We let $0\leq S\leq T$ and
consider the domains $\Psi A_{\left[-T,-S\right]}\Phi$, whose volume
equals $\frac{1}{2\rho}\cdot\mn\left(\Psi\right)\mu_{K}\left(\Phi\right)\left(e^{2\rho T}-e^{2\rho S}\right)$.
Note that when $S$ is increasing as a function of $T$, e.g. when
$S=\a T$ with $0<\a<1$, the domains in the family $\left\{ \Psi A_{\left[-T,-S\left(T\right)\right]}\Phi\right\} _{T>0}$
may not be contained in one another.
\begin{cor}
\label{cor: S in =00005B-T,T=00005D}Let $0\leq S\leq T$. For any
lattice $\gam<G$ and $\Psi$, $\Phi$, $\kappa$ as in Theorem \ref{thm: Counting in rectangles},
\begin{eqnarray*}
\#\left(\Psi A_{\left[-T,-S\right]}\Phi\cap\gam\right) & = & \frac{\mu\left(\Psi A_{\left[-T,-S\right]}\Phi\right)}{\mu\left(G/\gam\right)}+O\left(\left(\log\left(\mu\left(\rectlow T\left(\Psi,\Phi\right)\right)\right)\right)\cdot\mu\left(\rectlow T\left(\Psi,\Phi\right)\right)^{\kappa}\right)\\
 & = & \frac{\mn\left(\Psi\right)\mu_{K}\left(\Phi\right)}{\mu\left(G/\gam\right)}\cdot\frac{e^{2\rho T}-e^{2\rho S}}{2\rho}+O\left(T\left(e^{2\rho T}\right)^{\kappa}\right).
\end{eqnarray*}
\end{cor}
Note that the error term does not depend on $S$, and that when $S,T$
are such that $e^{2\rho T}-e^{2\rho S}=O\left(T\left(e^{2\rho T}\right)^{\kappa}\right)$,
i.e. when the interval $\left[-T,-S\right]$ is of length $O\left(e^{-2\rho\left(1-\kappa\right)T}\right)$,
the main term may be smaller than the error term, so this result is
only an upper bound.
\begin{proof}
Let $\gam<G$ and $\kappa=\kappa\left(\gam\right)$ as in Theorem
\ref{thm: Counting in rectangles}. According to this theorem, the
function
\[
E\left(T\right):=\#\left(\rectlow T\left(\Psi,\Phi\right)\cap\gam\right)-\mu\left(\rectlow T\left(\Psi,\Phi\right)\right)
\]
is bounded by $O\left(Te^{2\rho T\kappa}\right)$. Namely, there exist
$C,T_{0}>0$ such that
\[
\left|E\left(T\right)\right|\leq CTe^{2\rho T\cdot\kappa}\mbox{ for all }T\geq T_{0}.
\]
It follows that there exists $C'>0$ such that
\[
\left|E\left(T\right)\right|\leq C'Te^{2\rho T\cdot\kappa}\mbox{ for all }T\geq0.
\]
In particular, for $0\leq S\leq T$
\[
\left|E\left(S\right)\right|\leq C'\cdot Se^{2\rho S\cdot\kappa}\leq C'Te^{2\rho T\cdot\kappa},
\]
and therefore $E\left(S\right)=O\left(Te^{2\rho T\kappa}\right)$.
Now,
\begin{eqnarray*}
\#\left(\Psi A_{\left[-T,-S\right]}\Phi\cap\gam\right) & = & \#\left(\rectlow T\left(\Psi,\Phi\right)\cap\gam\right)-\#\left(\rectlow S\left(\Psi,\Phi\right)\cap\gam\right)\\
 & = & \frac{\mu\left(\rectlow T\left(\Psi,\Phi\right)\right)}{\mu\left(G/\gam\right)}+E\left(T\right)-\left(\frac{\mu\left(\rectlow S\left(\Psi,\Phi\right)\right)}{\mu\left(G/\gam\right)}+E\left(S\right)\right)\\
 & = & \frac{\mu\left(\rectlow T\left(\Psi,\Phi\right)\right)}{\mu\left(G/\gam\right)}-\frac{\mu\left(\rectlow S\left(\Psi,\Phi\right)\right)}{\mu\left(G/\gam\right)}+O\left(Te^{2\rho T\kappa}\right)\\
 & = & \frac{\mu\left(\Psi A_{\left[-T,-S\right]}\Phi\right)}{\mu\left(G/\gam\right)}+O\left(\left(\log\left(\mu\left(\rectlow T\left(\Psi,\Phi\right)\right)\right)\right)\cdot\mu\left(\rectlow T\left(\Psi,\Phi\right)\right)^{\kappa}\right).
\end{eqnarray*}
\end{proof}

\subsection{Lifts of horospheres}

Let $\gam$ be a non co-compact lattice in $G$, with a cusp at the
point $\sigma$ at the boundary of the associated hyperbolic space.
Let $H_{\sigma}$ be the unipotent subgroup in $G$ which stabilizes
$\sigma$ (in particular, it is conjugated to $N$). We consider the
case in which $\gam\cap H_{\sigma}$ is a lattice in $H_{\sigma}$.
Let ${\cal H}$ be a horosphere in the hyperbolic space of $G$ which
is based at $\sigma$; in other words, ${\cal H}$ is an orbit of
$H_{\sigma}$. Observe that ${\cal H}$ projects to a closed horosphere
$\overline{{\cal H}}$ in the space $\gam\backslash G$. Let $\ball_{T}\left(\point\right)$
denote a hyperbolic ball of radius $T$ that is centered at $\point$,
and let $N\left(T\right)$ denote the number of horospheres of the
form $\ga{\cal H}$,$\ga\in\gam$, that meet the ball $\ball_{T}\left(\point\right)$.
Eskin and McMullen \cite[Theorem 7.2]{EM93} have considered the counting
function $N\left(T\right)$ and discussed the case of $G=\mbox{PSL}_{2}\left(\RR\right)$.
This problem can be formulated for a Lie group of any real rank, see
\cite{Mohammadi_Golsefidy}; we will provide an effective estimate
for real rank $1$.
\begin{thm}
Let $\gam<G$ a non co-compact lattice, and let $\s$, $H_{\s}$,
${\cal H}$ as above. If $\gam\cap H_{\sigma}$ is a lattice in $H_{\sigma}$,
then
\[
N\left(T\right)=\frac{\mbox{Vol}_{\gam\backslash G}\left({\cal \overline{H}}\right)}{\mu\left(\gam\backslash G\right)}\cdot\frac{e^{2\rho T}}{2\rho}+O\left(T\left(e^{2\rho T}\right)^{\kappa}\right),
\]
where $\kappa=\kappa\left(\gam\right)$ is the exponent associated
with $\gam$.
\end{thm}
\begin{proof}
By conjugation, we may assume that $\point=i$ (the point stabilized
by $K$) and that $\sigma=\infty$ (the point stabilized by $N$),
namely $H_{\sigma}=N$. Then ${\cal H}$ is a horizontal horosphere,
i.e. it is orthogonal to the geodesic $A\cdot i$, and we may write
${\cal H}=Na_{y}\cdot i$ for some $y\in\RR$. Then the number of
horospheres $\ga{\cal H}$ that meet the ball $\ball_{T}\left(i\right)$
are in one to one correspondence with the elements of the following
set:
\[
\left\{ \ga N:d\left(i,\ga{\cal H}\right)<T\right\} =\left\{ \ga N:d\left(i,\ga Na_{y}\cdot i\right)<T\right\} .
\]
We write the elements of $\gam$ in their $KAN$ coordinates, and
denote $\ga=k_{\ga}a_{t\left(\ga\right)}n_{\ga}$.
\begin{eqnarray*}
=\left\{ \ga N:d\left(i,k_{\ga}a_{t\left(\ga\right)}n_{\ga}\,Na_{y}\cdot i\right)<T\right\}  & = & \left\{ \ga N:d\left(i,a_{t\left(\ga\right)}Na_{y}\cdot i\right)<T\right\} \\
 & = & \left\{ \ga N:d\left(i,a_{t\left(\ga\right)}N\cdot a_{-t\left(\ga\right)}\,a_{t\left(\ga\right)}\cdot a_{y}\cdot i\right)<T\right\} \\
 & = & \left\{ \ga N:d\left(i,N\,a_{t\left(\ga\right)+y}\cdot i\right)<T\right\} \\
 & = & \left\{ \ga N:d\left(i,a_{t\left(\ga\right)+y}\cdot i\right)<T\right\} ,
\end{eqnarray*}
since the horosphere $N\,a_{y+t\left(\ga\right)}\cdot i$ is orthogonal
to the geodesic $A\cdot i$, thus the point nearest to $i$ on this
horosphere is its meeting point with the geodesic, $a_{y+t\left(\ga\right)}\cdot i$.

Now, $d\left(i,a_{t\left(\ga\right)+y}\cdot i\right)=\left|t\left(\ga\right)+y\right|$,
so $d\left(i,a_{t\left(\ga\right)+y}\cdot i\right)<T$ if and only
if $-T-y\leq t\left(\ga\right)\leq T-y$. Moreover, the cosets $\ga N$
are in one to one correspondence with the lattice elements $\ga=k_{\ga}a_{t\left(\ga\right)}n_{\ga}$
such that $n_{\ga}\in\Psi\left(\gam\right)$, for a choice $\Psi\left(\gam\right)$
of a fundamental domain for $\gam\cap N$ in $N$. Then,
\begin{eqnarray*}
N\left(T\right) & = & \#\left\{ \ga=k_{\ga}a_{t\left(\ga\right)}n_{\ga}:n_{\ga}\in\Psi\left(\gam\right),-T-y\leq t\left(\ga\right)\leq T-y\right\} \\
 & = & \#\gam\cap\left(KA_{\left[-T-y,T-y\right]}\Psi\left(\gam\right)\right).
\end{eqnarray*}
Now the desired result follows from Theorem \ref{thm: Counting in rectangles}
and Remark \ref{rem: KAN vs NAK}:
\begin{eqnarray*}
N\left(T\right) & = & \frac{\mn\left(\Psi\left(\gam\right)\right)}{\mu\left(G/\gam\right)}\cdot\frac{e^{2\rho\left(T-y\right)}}{2\rho}+O\left(\left(T-y\right)\left(e^{2\rho\left(T-y\right)}\right)^{\kappa}\right)\\
 & = & \frac{\mn\left(\Psi\left(\gam\right)\right)}{\mu\left(G/\gam\right)}\cdot e^{-2\rho y}\cdot\frac{e^{2\rho T}}{2\rho}+O\left(T\left(e^{2\rho\left(T-y\right)}\right)^{\kappa}\right)\\
 & = & \frac{\mbox{Vol}_{\gam\backslash G}\left(\overline{{\cal H}}\right)}{\mu\left(\gam\backslash G\right)}\cdot\frac{e^{2\rho T}}{2\rho}+O\left(T\left(e^{2\rho T}\right)^{\kappa}\right).
\end{eqnarray*}
\end{proof}

\subsection{Diophantine equation associated with the Lorentz form }

When $G$ is $\mbox{SO}^{0}\left(1,n\right)$, $\mbox{SU}\left(1,n\right)$,
or $\mbox{SP}\left(1,n\right)$ (not just locally isomorphic to it),
then the elements of the subgroups $A$ and $N$ of $G$ can be written
explicitly as
\begin{equation}
a_{t}=\left(\begin{array}{ccc}
\cosh t & 0 & \sinh t\\
0 & I_{n-2} & 0\\
\sinh t & 0 & \cosh t
\end{array}\right)\label{eq: a_t}
\end{equation}
and
\begin{equation}
n_{v,z}=\left(\begin{array}{ccc}
1+z+\frac{1}{2}\left\Vert v\right\Vert ^{2} & v^{*} & -z-\frac{1}{2}\left\Vert v\right\Vert ^{2}\\
v & I_{n-2} & -v\\
z+\frac{1}{2}\left\Vert v\right\Vert ^{2} & v^{*} & 1-z-\frac{1}{2}\left\Vert v\right\Vert ^{2}
\end{array}\right)\label{eq: n_(v,z)}
\end{equation}
(e.g. \cite[p.373 and p.375]{Faraut}).

The explicit $N$ and $A$ components of a given $g\in G$ are extracted
in the following claim.
\begin{claim}
\label{claim: Extracting N and A components}Let
\[
g=\left(\begin{array}{ccc}
g_{0,0} & \cdots & g_{0,n}\\
\vdots &  & \vdots\\
g_{n,0} & \cdots & g_{n,n}
\end{array}\right)\in G.
\]
If $g=n_{v,z}a_{t}k$, then
\begin{eqnarray*}
e^{t} & = & \left(g_{0,0}-g_{n,0}\right)^{-1}\\
v & = & \frac{1}{g_{0,0}-g_{n,0}}\left(\begin{array}{c}
g_{1,0}\\
\vdots\\
g_{n-1,0}
\end{array}\right)\\
z & = & \frac{1}{2}\left(\frac{g_{0,0}+g_{n,0}}{g_{0,0}-g_{n,0}}-\frac{1+\sum_{j=1}^{n-1}\left|g_{j,0}\right|^{2}}{\left(g_{0,0}-g_{n,0}\right)^{2}}\right).
\end{eqnarray*}
\end{claim}
\begin{proof}
On the one hand,
\[
g\cdot i=\left(\begin{array}{ccc}
g_{0,0} & \cdots & g_{0,n}\\
\vdots &  & \vdots\\
g_{n,0} & \cdots & g_{n,n}
\end{array}\right)\left(\begin{array}{c}
1\\
0\\
\vdots\\
0
\end{array}\right)=\left(\begin{array}{c}
g_{0,0}\\
\vdots\\
g_{n,0}
\end{array}\right).
\]
 On the other hand,
\[
g\cdot i=n_{v,z}a_{t}k\cdot i=n_{v,z}a_{t}\cdot i,
\]
where
\begin{eqnarray*}
n_{v,z}a_{t}\cdot i & = & \left(\begin{array}{ccc}
1+z+\frac{1}{2}\left\Vert v\right\Vert ^{2} & v^{*} & -z-\frac{1}{2}\left\Vert v\right\Vert ^{2}\\
v & I_{n-2} & -v\\
z+\frac{1}{2}\left\Vert v\right\Vert ^{2} & v^{*} & 1-z-\frac{1}{2}\left\Vert v\right\Vert ^{2}
\end{array}\right)\left(\begin{array}{ccc}
\cosh t & 0 & \sinh t\\
0 & I_{n-2} & 0\\
\sinh t & 0 & \cosh t
\end{array}\right)\left(\begin{array}{c}
1\\
0\\
\vdots\\
0
\end{array}\right)\\
 & = & \left(\begin{array}{c}
\cosh t+e^{-t}\left(z+\frac{1}{2}\left\Vert v\right\Vert ^{2}\right)\\
e^{-t}\cdot v\\
\sinh t+e^{-t}\left(z+\frac{1}{2}\left\Vert v\right\Vert ^{2}\right)
\end{array}\right).
\end{eqnarray*}
Namely,
\[
\left(\begin{array}{c}
g_{0,0}\\
\vdots\\
g_{n,0}
\end{array}\right)=\left(\begin{array}{c}
\cosh t+e^{-t}\left(z+\frac{1}{2}\left\Vert v\right\Vert ^{2}\right)\\
e^{-t}\cdot v\\
\sinh t+e^{-t}\left(z+\frac{1}{2}\left\Vert v\right\Vert ^{2}\right)
\end{array}\right).
\]
In particular,
\[
g_{0,0}-g_{n,0}=\cosh t-\sinh t=e^{-t},
\]
and
\[
\left(\begin{array}{c}
g_{1,0}\\
\vdots\\
g_{n-1,0}
\end{array}\right)=e^{-t}\cdot v.
\]
Clearly $e^{t}$ and $v$ may be extracted from the above, and $z$
can be extracted from:
\begin{eqnarray*}
g_{0,0}+g_{n,0} & = & \cosh t+\sinh t+2e^{-t}\left(z+\frac{1}{2}\left\Vert v\right\Vert ^{2}\right)\\
 & = & e^{t}+2e^{-t}\left(z+\frac{1}{2}\left\Vert v\right\Vert ^{2}\right).
\end{eqnarray*}
\end{proof}
Let us focus on the case $G=\mbox{SO}^{0}\left(1,n\right)$. Clearly,
if $g=\left(g_{i,j}\right)_{0\leq i,j\leq n}$ is in $\mbox{SO}^{0}\left(1,n\right)$,
then $g_{0,0}^{2}-g_{1,0}^{2}-\cdots-g_{n,0}^{2}=1$, namely the first
column of $g$ satisfies the equation
\begin{equation}
x_{0}^{2}-x_{1}^{2}-\cdots-x_{n}^{2}=1.\label{eq: Lorentz equation}
\end{equation}
By Claim \ref{claim: Extracting N and A components} above, the $N$
and $A$ components of $g$ depend only on the first column of $g$;
hence, Corollary \ref{cor: Equidistribution of N and K} concerning
the equidistribution of the $N$-components as the $A$-components
approach $\infty$, can be used to study the behavior of the corresponding
parameters of equation \ref{eq: Lorentz equation}.

For every $\underline{x}=\left(x_{0},\ldots,x_{n}\right)\in\RR^{n+1}$
that satisfies equation \ref{eq: Lorentz equation}, define the height
function $h\left(\underline{x}\right)=\log\left(\frac{1}{x_{0}-x_{n}}\right)$
(corresponds to the $A$-component) and the vector $v\left(\underline{x}\right)=\frac{1}{x_{0}-x_{n}}\left(x_{1},\ldots,x_{n-1}\right)\in\RR^{n-1}$
(corresponds to the $N$-component). Assume $\Psi\subset\RR^{n+1}$
is nice. By applying Corollary \ref{cor: Equidistribution of N and K}
(part \ref{enu: N e.d.}) to the lattice $\mbox{SO}^{0}\left(1,n\right)\left(\Z\right)$
in $\mbox{SO}^{0}\left(1,n\right)$, we conclude
\begin{cor}
Consider the integral solutions $\underline{x}$ for $x_{0}^{2}-x_{1}^{2}-\cdots-x_{n}^{2}=1$.
The rational vectors $v\left(\underline{x}\right)$ become effectively
equidistributed in $\Psi$ as $-h\left(\underline{x}\right)\ra\infty$,
at rate $O\left(e^{-2\rho h\left(\underline{x}\right)\left(1-\kappa\right)}\cdot h\left(\underline{x}\right)\right)$.
\end{cor}
In the case where the lattice $\mbox{SO}^{0}\left(1,n\right)\left(\Z\right)$
is non-cocompact, i.e. when the Lorentz form is isotropic \cite{LS_10},
the equidistribution also occurs when $\left|h\left(\underline{x}\right)\right|\ra\infty$.

In analogy with the discussion in Section \ref{sec: Number-theoretical applications},
we may consider equidistribution of rational vectors $v\left(\underline{x}\right)$
that correspond to shortest integral representatives to cosets of
$N\left(\ZZ\right)$. Consider the discrete subgroup of $\RR^{n-1}$:
$\Lambda=\left\{ \left(x_{1},\ldots,x_{n-1}\right)\in\Z^{n-1}:\sum x_{i}\in2\Z\right\} $.
Then $N\left(\Z\right):=N\cap\mbox{SO}^{0}\left(1,n\right)\left(\Z\right)$
satisfies $N\left(\Z\right)=\left\{ n_{v}:v\in\Lambda\right\} $ by
Formula \ref{eq: n_(v,z)}, and is therefore isomorphic to $\Lambda$.
One possible choice for a fundamental domain for $\Lambda$ in $\RR^{n-1}$
is the unit ball with respect to the $\left\Vert \cdot\right\Vert _{1}$
norm, $\Psi_{0}:=\left\{ \left(x_{1},\ldots,x_{n-1}\right):\sum\left|x_{i}\right|\leq1\right\} $.
This fundamental domain has the property that it contains the shortest
(with respect to the $2$-norm) representative of every coset of $\Lambda$.

For every $h\in\RR$, let $X_{h}$ denote the set of solutions with
height $h$; for every $h$ such that $X_{h}\cap\Z^{n+1}\neq\emptyset$,
let $x_{h}$ denote the unique integral solution for which $v\left(\underline{x}_{h}\right)$
is the shortest, namely lies in $\Psi_{0}$. By applying Corollary
\ref{cor: Equidistribution of N and K} to the lattice $\mbox{SO}^{0}\left(1,n\right)\left(\Z\right)$
in $\mbox{SO}^{0}\left(1,n\right)$, $\Psi=\Psi_{0}$ and $\Phi=K$,
we conclude:
\begin{cor}
The rational vectors $v\left(\underline{x}_{h}\right)$ associated
with shortest solutions, become effectively equidistributed (at rate
as above) in $\Psi_{0}$ as $\left|h\right|\ra\infty$.
\end{cor}
Compare to part \ref{enu: e.d. of quotients of norms} of Theorem
\ref{thm: R-R extended in integers}. Of course, this can be done
for any Lorentz form defined over $\Q$.

\section{Proof of the main theorem\label{sec: Proof of Thm Rectangles}}

\subsection{A spectral method for counting lattice points\label{subsec: The GN method}}

In the following discussion, $G$ is an almost simple Lie group, not
necessarily of rank $1$. The lattice point counting method in family
of domains $\left\{ \B_{T}\right\} \subset G$ that we will use (\cite{GN1},
\cite{GN_book}) has two ingredients: a spectral estimate and a regularity
property. The crucial spectral estimate requires bounding the norm
of the averaging operators defined by $\B_{T}$ in the representation
on $L_{0}^{2}\left(\gam\backslash G\right)$. Let us recall the fact
that there exists $m\in\mathbb{N}$ such that the unitary representation
of $G$ in $L_{0}^{2}\left(\gam\backslash G\right)$, when taken to
the $m$-th tensor power, is weakly contained in the regular representation
of $G$. The essential property of such $m$ is that $m\geq p/2$,
where $p$ satisfies that the $K$-finite matrix coefficients of $\pi_{\gam\backslash G}^{0}$
are in $L^{p+\e}\left(G\right)$ for every $\e>0$. We define $m\left(\gam\right)$
to be the least even integer with this property if $p>2$, or $1$
if $p=2$ (see \cite[Definition 3.1]{GN1}). One of the remarkable
features of harmonic analysis on simple Lie groups is that then for
any measurable set of positive finite measure $B$ in $G$, if we
denote by $\b$ the Haar uniform measure on $B$, the following estimate
holds \cite{N98}:
\begin{equation}
\mbox{for every }\ee>0\mbox{, }\left\Vert \pi_{\gam\backslash G}^{0}\left(\b\right)\right\Vert \leq C_{G,\ee}\cdot m_{G}\left(B\right)^{-1/2m\left(\gam\right)+\ee}.\label{eq: decay of operator norm}
\end{equation}
Thus, $m\left(\gam\right)$ measures the size of the spectral gap
in $L^{2}\left(\gam\backslash G\right)$. The lattice $\gam$ is called
\emph{tempered} if the representation $\pi_{\gam\backslash G}^{0}$
is already weakly contained in regular representation, namely if $m\left(\gam\right)=1$.

We now turn to the second ingredient, which is the Lipschitz property
of the domains $\B_{T}$.
\begin{defn}[\cite{GN1}]
\label{def: well--roundedness}Let $G$ be a Lie group with Haar
measure $m_{G}$. Assume $\left\{ \B_{T}\right\} \subset G$ is a
family of bounded domains of positive-measure such that $m_{G}\left(\B_{T}\right)\ra\infty$
as $T\ra\infty$. Let $\Ocal_{\e}\subset G$ be the image of a ball
of radius $\e$ in the Lie algebra under the exponential map. Denote
\[
\B_{T}^{+}\left(\e\right):=\Ocal_{\e}\B_{T}\Ocal_{\e}=\bigcup_{u,v\in\Ocal_{\e}}u\,\B_{T}\,v,
\]
\[
\B_{T}^{-}\left(\e\right):=\bigcap_{u,v\in\Ocal_{\e}}u\,\B_{T}\,v
\]
(Figure \ref{fig: Well-Roundedness}). The family $\left\{ \B_{T}\right\} $
is \emph{Lipschitz well-rounded} if there exist $\e_{0}>0$ and $T_{0}\geq0$
such that for every $0<\e\leq\e_{0}$ and $T\geq T_{0}$:
\[
m_{G}\left(\B_{T}^{+}\left(\e\right)\right)\leq\left(1+C\e\right)\,m_{G}\left(\B_{T}^{-}\left(\e\right)\right),
\]
where $C>0$ is a constant that does not depend on $\e$ or $T$.
\end{defn}
\begin{figure}
\hspace*{\fill}\subfloat[The set $\protect\B_{T}$]{\includegraphics[scale=0.5]{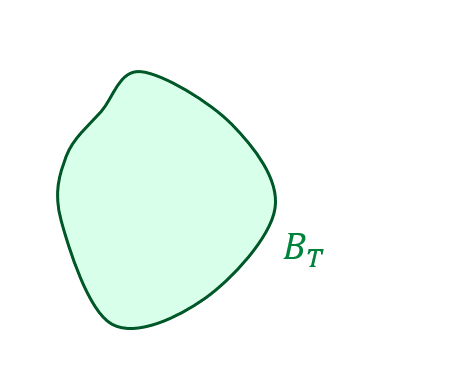}}\hspace*{\fill}\subfloat[The set $\protect\B_{T}$ is perturbed by ${\cal O}_{\protect\e}$]{\hspace*{5mm}\includegraphics[scale=0.5]{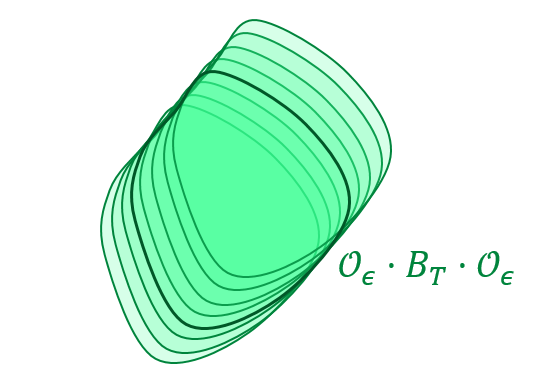}\hspace*{5mm}}\hspace*{\fill}\subfloat[$\protect\B_{T}^{-}\left(\protect\e\right)$ and $\protect\B_{T}^{+}\left(\protect\e\right)$]{\includegraphics[scale=0.5]{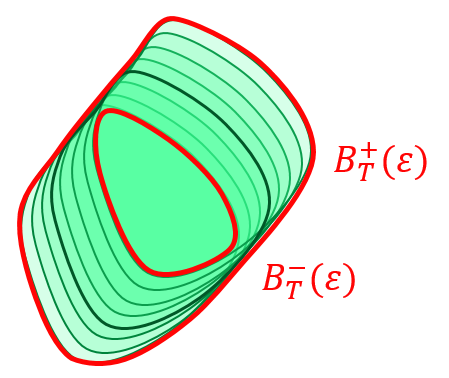}}\hspace*{\fill}

\caption{Well-roundedness.\label{fig: Well-Roundedness}}
\end{figure}
The concept of well-roundedness appeared
first in \cite{DRS93}, and later formulated in \cite{EM93}. It has
also been used in \cite{Gorodnik_Weiss}. The conditions in Definition
\ref{def: well--roundedness} generalize those that occurred in the
aforementioned papers.
\begin{thm}[\cite{GN1}]
\label{Thm: GN Counting thm} Let $G$ be an almost simple Lie group
with Haar measure $m_{G}$, and let $\gam<G$ be a lattice. Assume
$\left\{ \mathcal{B}_{T}\right\} \subset G$ is a family of finite-measure
domains which satisfies $m_{G}\left(\B_{T}\right)\ra\infty$ as $T\ra\infty$.
If the family $\left\{ \mathcal{B}_{T}\right\} $ is Lipschitz well-rounded,
then
\[
\#\left(\B_{T}\cap\gam\right)=\frac{1}{m_{G}\left(G/\gam\right)}m_{G}\left(\B_{T}\right)+O\left(m_{G}\left(\B_{T}\right)\cdot E\left(T\right)^{\frac{1}{1+\dim\left(G\right)}}\right)
\]
as $T\ra\infty$, where $m_{G}\left(G/\gam\right)$ is the measure
of a fundamental domain of $\gam$ in $G$, and $E\left(T\right)$
is (a bound on) the rate of decay of operator norm $\left\Vert \pi_{\gam\backslash G}^{0}\left(\b_{T}\right)\right\Vert $.
\end{thm}
Note that the above theorem applies to every lattice $\gam$.

When plugging the estimation \ref{eq: decay of operator norm} for
$\left\Vert \pi_{\gam\backslash G}^{0}\left(\b_{T}\right)\right\Vert $,
the obtained error term in Theorem \ref{Thm: GN Counting thm} is:
\[
O\left(m_{G}\left(\B_{T}\right)^{\kappa\left(\gam\right)+\e}\right),\mbox{ for every }\ee>0
\]
where
\begin{equation}
\kappa\left(\gam\right)=1-\frac{1}{2m\left(\gam\right)\left(1+\dim\left(G\right)\right)}\in\left(0,1\right).\label{eq: kappa lattice exponent}
\end{equation}

In our case, where $G$ is of real rank one and the family of domains
is $\rectlow T\left(\Psi,\Phi\right)$, the estimation \ref{eq: decay of operator norm}
may be improved so that the error term is reduced to
\[
O\left(\left(\log\left(m_{G}\left(\rectlow T\left(\Psi,\Phi\right)\right)\right)\cdot m_{G}\left(\rectlow T\left(\Psi,\Phi\right)\right)\right)^{\kappa\left(\gam\right)}\right),
\]
as we now explain. Assume that a set $B\subset G$ of positive finite
measure satisfies that
\[
\mu\left(K\cdot B\cdot K\right)\leq\mbox{const}\cdot\mu\left(B\right);
\]
this property is called $K$-radializability (\cite[Def. 3.21]{GN_book}).
When $B$ is radializable, then it is a consequence of the spectral
transfer principle \cite{N98} and of estimates on the Harish-Chandra
function in real rank one that
\begin{eqnarray*}
\left\Vert \pi_{\gam\backslash G}^{0}\left(\b\right)\right\Vert  & \leq & C_{G}\cdot\left(\log\left(\mu\left(B\right)\right)\right)^{\frac{1}{m\left(\gam\right)}}\cdot\left(\mu\left(B\right)\right)^{-\frac{1}{2m\left(\gam\right)}}\\
 & \leq & C_{G}\cdot\left(\log\left(\mu\left(B\right)\right)\right)\cdot\left(\mu\left(B\right)\right)^{-\frac{1}{2m\left(\gam\right)}}
\end{eqnarray*}
(\cite[Prop. 5.9]{GN_book}). The sets $\rectlow T\left(\Psi,\Phi\right)$
are indeed radializable, with constant that does not depend on $T$.
Thus, if $\b_{T}$ are the probability measures that corresponds to
$\rectlow T=\rectlow T\left(\Psi,\Phi\right)$, then
\[
E\left(T\right)=\left\Vert \pi_{\gam\backslash G}^{0}\left(\b_{T}\right)\right\Vert \leq C_{G}\cdot\left(\log\left(\mu\left(\rectlow T\right)\right)\right)\cdot\left(\mu\left(\rectlow T\right)\right)^{-\frac{1}{2m\left(\gam\right)}},
\]
as claimed.

From the above discussion it follows that \emph{in order to prove
Theorem \ref{thm: Counting in rectangles}, it suffices to show that
the family $\left\{ \rectlow T\left(\Psi,\Phi\right)\right\} $ is
Lipschitz well rounded}.

\subsection{Lipschitz property for Iwasawa coordinates in the negative direction
of $A$}

\begin{table}
\centering{}%
\begin{tabular}{|c|c|c|c|c|}
\hline
$G$ & $\mbox{SO}^{0}\left(1,n\right)$ & $\mbox{SU}\left(1,n\right)$ & $\mbox{SP}\left(1,n\right)$ & $\mbox{F}_{4\left(-20\right)}$\tabularnewline
\hline
\hline
$\K$ & $\RR$ & $\C$ & $\H$ & $\mathbb{O}$\tabularnewline
\hline
$N$ (as a manifold) & $\RR^{n-1}$ & $\C^{n-1}\oplus\RR$ & $\H^{n-1}\oplus\RR^{3}$ & $\mathbb{O}\oplus\RR^{7}$\tabularnewline
\hline
$K$ & $\mbox{SO}\left(n\right)$ & $\mbox{S}\left(\mbox{U}\left(1\right)\times\mbox{U}\left(n\right)\right)$ & $\mbox{SP}\left(1\right)\times\mbox{SP}\left(n\right)$ & $\mbox{Spin}\left(9\right)$\tabularnewline
\hline
$G/K$ & $\hyp_{\RR}^{n}$ & $\hyp_{\C}^{n}$ & $\hyp_{\H}^{n}$ & $\hyp_{\mathbb{O}}^{2}$.\tabularnewline
\hline
$\left(p,q\right)$ & $\left(n-1,0\right)$ & $\left(2n-2,1\right)$ & $\left(4n-4,3\right)$ & $\left(8,7\right)$\tabularnewline
\hline
$\mu$ & $\mn\times\frac{dt}{e^{\left(n-1\right)t}}\times\mu_{K}$ & $\mn\times\frac{dt}{e^{2nt}}\times\mu_{K}$ & $\mn\times\frac{dt}{e^{\left(4n+2\right)t}}\times\mu_{K}$ & $\mn\times\frac{dt}{e^{22t}}\times\mu_{K}$\tabularnewline
\hline
\end{tabular}\caption{Simple rank $1$ Lie groups: Iwasawa subgroups, symmetric spaces and
Haar measure\label{tab: rank 1 lie groups}}
\end{table}

In order to show that the family $\rectlow T\left(\Psi,\Phi\right)$
is Lipschitz well-rounded, it will be convenient to introduce coordinates
on $N$ as well, in addition to the parametrization we have already
set for $A$; recall $A=\left\{ a_{t}:t\in\RR\right\} $ such that
$d\left(a_{t}\cdot i,a_{s}\cdot i\right)=\left|t-s\right|$. Let $\mathbb{K}\in\left\{ \RR,\C,\H\right\} $
be the ``field'' over which the matrices in $G$ are defined, and
$n$ the dimension (over $\K$) of the corresponding hyperbolic space.
The group $N$ is of Heisenberg type (see \cite{H_type_91}, \cite{H_type_98}),
and in particular it is parametrized by the space $\mathbb{K}^{n}\oplus\im\left(\mathbb{K}\right)$,
where $\im\left(\mathbb{K}\right)$ is the subspace of ``pure imaginary''
numbers in $\K$, namely of elements $w$ such that $w+\bar{w}=0$.
A parametrization may be chosen such that
\[
N=\left\{ n_{v,z}:v\in\mathbb{K}^{n},z\in\im\left(\mathbb{K}\right)\right\} ,
\]
 with the group multiplication
\[
n_{v_{1},z_{1}}n_{v_{2},z_{2}}=n_{v_{1}+v_{2},z_{1}+z_{2}+\im\left(\left\langle v_{2},v_{1}\right\rangle \right)}
\]
(where $\left\langle v_{2},v_{1}\right\rangle =v_{1}^{*}v_{2}$).The
subspaces $\mathbb{K}^{n}$ and $\im\left(\mathbb{K}\right)$ correspond
to subsets of $N$ that are invariant under conjugation by $A$, and
specifically,
\begin{equation}
a_{t}\,n_{v,z}\,a_{-t}=n_{e^{t}v,e^{2t}z}.\label{eq: conjugation of N by A}
\end{equation}
As a result, if $p:=\dim_{\RR}\left(\mathbb{K}^{n}\right)$ and $q:=\dim_{\RR}\left(\im\left(\mathbb{K}\right)\right)=\dim_{\RR}\left(\mathbb{K}\right)-1$,
then $\mu_{N}$ is the Lebesgue measure on $\RR^{p+q}$, and the parameter
$\rho$ that appears in Formula \ref{eq: Haar measure on G} for the
Haar measure equals $\frac{1}{2}\left(p+2q\right)$.

Let $\overline{N}$ denote the opposite unipotent group, namely the
one that corresponds to the negative roots:
\begin{equation}
a_{t}\,\overline{n}_{v,z}\,a_{-t}=n_{e^{-t}v,e^{-2t}z}.\label{eq: conj of opposite n}
\end{equation}
On the subgroups $H\in\left\{ A,K\right\} $ we consider the metric
$d_{H}$ induced by the Riemannian metric on $G$. We denote by $K_{\left(\phi,\dl\right)}$
a ball in $K$ with center $\phi\in K$ and radius $\dl$, and by
$A_{\left(t,\dl\right)}$ a ball in $A$, with center $t$ and radius
$\dl$ (these are simply the elements that correspond to the interval
$\left(t-\dl,t+\dl\right)$, since $d_{A}$ is the euclidean metric
on $\RR$). We let $d_{N}$ denote the product of euclidean metrics
on $\mathbb{K}^{n}\cong\RR^{p}$ and $\im\left(\mathbb{K}\right)\cong\RR^{q}$,
and let $N_{\left(v,\dl_{1}\right)\times\left(z,\dl_{2}\right)}$
be the domain in $N$ parametrized by the product of euclidean balls
in $\mathbb{K}^{n}\cong\RR^{p}$ and $\im\left(\mathbb{K}\right)\cong\RR^{q}$
with centers $v,z$ and radii $\dl_{1},\dl_{2}$ respectively. When
a ball is centered at the identity we omit the center and denote $K_{\left(\dl\right)}$,
$A_{\left(\dl\right)}$, and $N_{\left(\dl_{1}\right)\times\left(\dl_{2}\right)}$.

In what follows, $\left\Vert \cdot\right\Vert _{\mbox{ck}}$ is the
Cartan-Killing norm on the Lie algebra $\mbox{Lie}\left(G\right)$
of $G$, and $\left\Vert \cdot\right\Vert _{\mbox{op}}$ is the norm
on the space of linear operators on $\mbox{Lie}\left(G\right)$.
\begin{lem}
\label{lem: Conjugation inflates by the norm of Ad} Let $G$ be a
semi-simple linear Lie group. Let $\ball_{\e}=\left\{ X\in\mbox{Lie}\left(G\right):\left\Vert X\right\Vert \leq\e\right\} $,
and let ${\cal O}_{\e}=\exp\left(\ball_{\e}\right)$. For every $g\in G$,
\[
g^{-1}\,{\cal O}_{\e}\,g\subseteq{\cal O}_{\e\cdot\left\Vert \mbox{Ad}\,g\right\Vert _{\mbox{op}}}=\exp\left\{ X\in\mbox{Lie}\left(G\right):\left\Vert X\right\Vert _{\mbox{ck}}\leq\e\cdot\left\Vert \mbox{Ad}\,g\right\Vert _{\mbox{op}}\right\} .
\]
\end{lem}
\begin{proof}
Recall
\[
\left\Vert \mbox{Ad}\,g\right\Vert _{\mbox{op}}=\underset{X\in\ball_{1}}{\max}\left\Vert \mbox{Ad\,}g\left(X\right)\right\Vert _{\mbox{ck}}=\underset{X\in\ball_{1}}{\max}\left\Vert g^{-1}Xg\right\Vert _{\mbox{ck}}.
\]
Observe that $\mbox{Ad}\,g\left(\ball_{\e}\right)\subset\mbox{Lie}\left(G\right)$
is contained in a ball of radius
\[
\underset{X\in\ball_{\e}}{\max}\left\Vert \mbox{Ad}\,g\left(X\right)\right\Vert _{\mbox{ck}}=\underset{X\in\ball_{\e}}{\max}\left\Vert g^{-1}Xg\right\Vert _{\mbox{ck}}=\underset{X\in\e\ball_{1}}{\max}\left\Vert g^{-1}Xg\right\Vert _{\mbox{ck}}=
\]
\[
=\underset{X\in\ball_{1}}{\max}\left\Vert g^{-1}\e Xg\right\Vert _{\mbox{ck}}=\e\left\Vert \mbox{Ad\,}g\right\Vert _{\mbox{op}}.
\]
Now,
\[
g^{-1}\,{\cal O}_{\e}\,g=g^{-1}\exp\left(\ball_{\e}\right)g=\exp\left(g^{-1}\,\ball_{\e}\,g\right)=\exp\left(\mbox{Ad}g\left(\ball_{\e}\right)\right)
\]
\[
\subseteq\exp\left(\ball_{\e\cdot\left\Vert \mbox{Ad}\,g\right\Vert _{\mbox{op}}}\right)={\cal O}_{\e\cdot\left\Vert \mbox{Ad}\,g\right\Vert _{\mbox{op}}}.
\]
\end{proof}
Let $M$ denote the centralizer of $A$ in $K$. We will use the following:
there exists $\dl_{0}>0$ such that for all $0<\dl\leq\dl_{0}$, there
are positive constants $c_{1},c_{2}$ such that
\begin{equation}
{\cal O}_{\dl}\subseteq N_{\left(c_{1}\dl\right)\times\left(c_{1}\dl\right)}A_{\left(c_{1}\dl\right)}K_{\left(c_{1}\dl\right)},\label{eq: Ball contained in Iwasawa box}
\end{equation}
\begin{equation}
{\cal O}_{\dl}\subseteq N_{\left(c_{2}\dl\right)\times\left(c_{2}\dl\right)}A_{\left(c_{2}\dl\right)}M_{\left(c_{2}\dl\right)}\overline{N}_{\left(c_{2}\dl\right)\times\left(c_{2}\dl\right)}\label{eq: Ball contained in NMAN box}
\end{equation}
(the latter are the Bruhat coordinates on a neighborhood of the identity
in $G$).
\begin{prop}[Effective Iwasawa decomposition]
\label{prop: Effective Iwasawa decomposition} Let $n_{v,z}\in N$,
$\phi\in K$, $a_{t}\in A$ with $t\leq0$. There exists $\e_{1}>0$
such that for every $0<\e\leq\e_{1}$ there are positive constants
$C_{N}',C_{N}'',C_{A},C_{K}$ that depend only on $n_{v,z}$ and $\phi$
(in particular, independent of $t$!) such that
\[
{\cal O}_{\e}\cdot n_{v,z}a_{t}\phi\cdot{\cal O}_{\e}\subset N_{\left(v,C_{N}'\e\right)\times\left(z,C_{N}''\e\right)}A_{\left(t,C_{A}\e\right)}K_{\left(\phi,C_{K}\e\right)}.
\]
Furthermore, when $n_{v,z}$ varies over a compact set $\Psi$, and
$\phi$ varies over $K$, these constants can be taken to be uniform.
\end{prop}
\begin{proof}
Observe that
\[
N_{\left(\dl_{1}\right)\times\left(\dl_{2}\right)}N_{\left(\rho_{1}\right)\times\left(\rho_{2}\right)}\subseteq N_{\left(\dl_{1}+\rho_{1}\right)\times\left(\dl_{2}+\rho_{2}+\rho_{1}\dl_{1}\right)}
\]
and
\begin{equation}
n_{v,z}N_{\left(\rho_{1}\right)\times\left(\rho_{2}\right)}\subseteq N_{\left(v,\rho_{1}\right)\times\left(z,\rho_{2}+\left\Vert v\right\Vert \rho_{1}\right)}.\label{eq: multiplication by N-ball}
\end{equation}
In particular,
\begin{eqnarray}
 &  & n_{v,z}N_{\left(\dl_{1}\right)\times\left(\dl_{2}\right)}N_{\left(\rho_{1}\right)\times\left(\rho_{2}\right)}\nonumber \\
 & \subseteq & n_{v,z}N_{\left(\dl_{1}+\rho_{1}\right)\times\left(\dl_{2}+\rho_{2}+\rho_{1}\dl_{1}\right)}\nonumber \\
 & \subseteq & N_{\left(v,\dl_{1}+\rho_{1}\right)\times\left(z,\dl_{2}+\rho_{2}+\rho_{1}\dl_{1}+\left\Vert v\right\Vert \left(\dl_{1}+\rho_{1}\right)\right)}\label{eq: multp by N-ball and N-point}
\end{eqnarray}
Finally, note that
\begin{equation}
K_{\left(\dl\right)}\phi\subset K_{\left(\phi,\dl\right)}.\label{eq: multiplication by K-ball}
\end{equation}

\paragraph*{Step 1: Right perturbations.}

We show that
\[
n_{v,z}a_{t}\phi\cdot{\cal O}_{\e}\subset N_{\left(v,r_{1}\e\right)\times\left(z,r_{2}\e\right)}A_{\left(t,r_{3}\e\right)}K_{\left(\phi,r_{4}\e\right)}
\]
where $r_{i}=r_{i}\left(v,z\right)$ is independent of $t\leq0$.
Recall $\left\Vert \mbox{Ad}\,\phi\right\Vert =1$. By Lemma \ref{lem: Conjugation inflates by the norm of Ad},
\[
n_{v,z}a_{t}\phi\cdot{\cal O}_{\e}\subseteq n_{v,z}\,a_{t}\,{\cal O}_{\e}\phi.
\]
By \ref{eq: Ball contained in Iwasawa box},
\[
\subseteq n_{v,z}\,a_{t}\left(N_{\left(c_{1}\e\right)\times\left(c_{1}\e\right)}A_{\left(c_{1}\e\right)}K_{\left(c_{1}\e\right)}\right)\phi.
\]
By \ref{eq: conjugation of N by A},
\[
\subseteq n_{v,z}N_{\left(e^{t}c_{1}\e\right)\times\left(e^{2t}c_{1}\e\right)}a_{t}A_{\left(c_{1}\e\right)}K_{\left(c_{1}\e\right)}\phi,
\]
and by \ref{eq: multiplication by N-ball} and \ref{eq: multiplication by K-ball},
\[
\subseteq N_{\left(v,c_{1}e^{t}\e\right)\times\left(z,c_{1}e^{2t}\e+c_{1}\left\Vert v\right\Vert e^{t}\e\right)}A_{\left(t,c_{1}\e\right)}K_{\left(\phi,c_{1}\e\right)}.
\]
Since $e^{t}\leq1$,
\[
\subseteq N_{\left(v,c_{1}\e\right)\times\left(z,c_{1}\e+c_{1}\left\Vert v\right\Vert \e\right)}A_{\left(t,c_{1}\e\right)}K_{\left(\phi,c_{1}\e\right)}.
\]

\paragraph*{Step 2: Left perturbations. }

We show that
\[
{\cal O}_{\e}\cdot n_{v,z}a_{t}\phi\subset N_{\left(v,\l_{1}\e\right)\times\left(z,\l_{2}\e\right)}A_{\left(t,\l_{3}\e\right)}K_{\left(\phi,\l_{4}\e\right)}
\]
where $\l_{i}=\l_{i}\left(v,z\right)$ is independent of $t\leq0$.

Denote $\eta=\left\Vert \mbox{Ad}\,n_{v,z}\right\Vert _{\mbox{op}}$.
By Lemma \ref{lem: Conjugation inflates by the norm of Ad},
\[
{\cal O}_{\e}\cdot n_{v,z}a_{t}\phi\subseteq n_{v,z}{\cal O}_{\eta\e}a_{t}\phi.
\]
Set $\e_{1}=\min\left\{ 1,\dl_{0}/\eta\right\} $. Then for $\e\leq\e_{1}$,
\ref{eq: Ball contained in NMAN box} implies
\[
\subseteq n_{v,z}\left(N_{\left(c_{2}\eta\e\right)\times\left(c_{2}\eta\e\right)}A_{\left(c_{2}\eta\e\right)}M_{\left(c_{2}\eta\e\right)}\overline{N}_{\left(c_{2}\eta\e\right)\times\left(c_{2}\eta\e\right)}\right)a_{t}\phi
\]
by \ref{eq: conj of opposite n},
\[
\subseteq n_{v,z}N_{\left(c_{2}\eta\e\right)\times\left(c_{2}\eta\e\right)}A_{\left(c_{2}\eta\e\right)}a_{t}\,M_{\left(c_{2}\eta\e\right)}\overline{N}_{\left(c_{2}e^{t}\eta\e\right)\times\left(c_{2}e^{2t}\eta\e\right)}\phi.
\]
Since $M\overline{N}_{\dl_{1},\dl_{2}}\subseteq K{\cal O}_{\max\left\{ \dl_{1},\dl_{2}\right\} }={\cal O}_{\max\left\{ \dl_{1},\dl_{2}\right\} }$
and $e^{t}\geq e^{2t}$,
\[
\subseteq n_{v,z}N_{\left(c_{2}\eta\e\right)\times\left(c_{2}\eta\e\right)}A_{\left(c_{2}\eta\e\right)}a_{t}\,\left({\cal O}_{c_{2}e^{t}\eta\e}\right)\phi.
\]
By \ref{eq: Ball contained in Iwasawa box},
\[
\subseteq n_{v,z}N_{\left(c_{2}\eta\e\right)\times\left(c_{2}\eta\e\right)}A_{\left(t,c_{2}\eta\e\right)}\left(N_{\left(c_{1}c_{2}e^{t}\eta\e\right)\times\left(c_{1}c_{2}e^{t}\eta\e\right)}A_{\left(c_{1}c_{2}e^{t}\eta\e\right)}K_{\left(c_{1}c_{2}e^{t}\eta\e\right)}\right)\phi
\]
and by \ref{eq: conjugation of N by A} and \ref{eq: multiplication by K-ball},
\[
\subseteq n_{v,z}N_{\left(c_{2}\eta\e\right)\times\left(c_{2}\eta\e\right)}N_{\left(e^{t+c_{2}\eta\e}\cdot c_{1}c_{2}e^{t}\eta\e\right)\times\left(e^{2\left(t+c_{2}\eta\e\right)}\cdot c_{1}c_{2}e^{t}\eta\e\right)}A_{\left(t,c_{2}\eta\e\right)}A_{\left(c_{1}c_{2}e^{t}\eta\e\right)}K_{\left(\phi,c_{1}c_{2}e^{t}\eta\e\right)}.
\]

\[
\subseteq N_{\left(v,\left(1+c_{1}e^{2t+c_{2}\eta\e}\right)c_{2}\eta\e\right)\times\left(z,\left(1+c_{1}e^{3t+2c_{2}\eta\e}+c_{1}e^{2t+c_{2}\eta\e}\eta\e+\left\Vert v\right\Vert +\left\Vert v\right\Vert c_{1}e^{2t+c_{2}\eta\e}\right)c_{2}\eta\e\right)}A_{\left(t,\left(1+c_{1}e^{t}\right)c_{2}\eta\e\right)}K_{\left(\phi,c_{1}c_{2}e^{t}\eta\e\right)}.
\]
Since $e^{t}\leq1$ and $\e\leq\e_{1}\leq1$,
\[
\subseteq N_{\left(v,\left(1+c_{1}e^{c_{2}\eta}\right)c_{2}\eta\e\right)\times\left(z,\left(1+c_{1}e^{2c_{2}\eta}+c_{1}e^{c_{2}\eta}\eta+\left\Vert v\right\Vert +\left\Vert v\right\Vert c_{1}e^{c_{2}\eta}\right)c_{2}\eta\e\right)}A_{\left(t,\left(1+c_{1}\right)c_{2}\eta\e\right)}K_{\left(\phi,c_{1}c_{2}\eta\e\right)}.
\]

\paragraph*{Step 3: Combining left and right perturbations. }

Let $g:=n_{v,z}a_{t}\phi$ with $t\leq0$ and let $\e\leq\e_{1}$.
Choose uniform (independent of $t$) constants $\overline{\l_{i}}=\max\left\{ \l_{i}\left(v',z'\right):n_{v',z'}\in\pi_{N}\left(g\cdot{\cal O}_{1}\right)\right\} $.
Since $g\cdot{\cal O}_{\e}\subset g\cdot{\cal O}_{1}$, it follows
from Step 2 that for every
\[
g_{0}=n_{v_{0},z_{0}}a_{t_{0}}\phi_{0}\in g\cdot{\cal O}_{\e},
\]
it holds that
\[
{\cal O}_{\e}\cdot g_{0}\subset N_{\left(v_{0},\overline{\l_{1}}\e\right)\times\left(z_{0},\overline{\l_{2}}\e\right)}A_{\left(t_{0},\overline{\l_{3}}\e\right)}K_{\left(\phi_{0},\overline{\l_{4}}\e\right)}.
\]
But, as was shown in Step 1, $d_{N}\left(v_{0},v\right)\leq r_{1}\e$,
$d_{N}\left(z_{0},z\right)\leq r_{2}\e$, $d_{A}\left(t_{0},t\right)\leq r_{3}\e$
and $d_{K}\left(\phi_{0},\phi\right)\leq r_{4}\e$. Then by the triangle
inequality,
\[
{\cal O}_{\e}\cdot g\cdot{\cal O}_{\e}\subset N_{\left(v,r_{1}\e+\overline{\l_{1}}\e\right)\times\left(z,r_{2}\e+\overline{\l_{2}}\e\right)}A_{\left(t,r_{3}\e+\overline{\l_{3}}\e\right)}K_{\left(\phi,r_{4}\e+\overline{\l_{4}}\e\right)}.
\]
\end{proof}

\subsection{Lipschitz-Regularity of the domains $\protect\rectlow T\left(\Psi,\Phi\right)$}

Recall that we wish to show that the family $\left\{ \rectlow T\left(\Psi,\Phi\right)\right\} _{T>0}$
is Lipschitz well-rounded (Definition \ref{def: well--roundedness}).
Since we have already established the Lipschitz property for the Iwasawa
coordinates in the negative direction of $A$, all that remains is
to bound the quotient of the measures of $\rectlow T\left(\Psi,\Phi\right)^{+}\left(\e\right)$
and $\rectlow T\left(\Psi,\Phi\right)^{-}\left(\e\right)$, which
we perform below.
\begin{proof}[Proof of Theorem \ref{thm: Counting in rectangles}]
 Throughout this proof, it will be convenient to parametrize $N$
as $\RR^{p+q}$ instead of $\RR^{p}\oplus\RR^{q}$. We will write
$n_{\underline{x}}$ instead of $n_{v,z}$, and $N_{\left(\underline{x},\dl\right)}$
for a ball of radius $\dl$ centered at $\underline{x}$.

For convenience, let us denote $\mu_{A}=\frac{dt}{e^{2\rho t}}$.
Then $\mu=\mu_{N}\times\mu_{A}\times\mu_{K}$ and therefore it is
sufficient to show that there exist $\e_{0},T_{0}>0$ such that for
every $H\in\left\{ N,A,K\right\} $ there exists a positive constant
$c_{H}$ satisfying
\[
\frac{\mu_{H}\left(\pi_{H}\left(\rectlow T\left(\Psi,\Phi\right)^{+}\left(\e\right)\right)\right)}{\mu_{H}\left(\pi_{H}\left(\rectlow T\left(\Psi,\Phi\right)^{-}\left(\e\right)\right)\right)}\leq1+c_{H}\e
\]
for every $0<\e\leq\e_{0}$ , $T\geq T_{0}$. Alternatively,
\[
\frac{\mu_{H}\left(\pi_{H}\left(\rectlow T\left(\Psi,\Phi\right)^{+}\left(\e\right)\right)\right)-\mu_{H}\left(\pi_{H}\left(\rectlow T\left(\Psi,\Phi\right)^{-}\left(\e\right)\right)\right)}{\mu_{H}\left(\pi_{H}\left(\rectlow T\left(\Psi,\Phi\right)^{-}\left(\e\right)\right)\right)}\leq c_{H}\e
\]
for every $0<\e\leq\e_{0}$ , $T\geq T_{0}$. Since this is a property
of the measures of $\pi_{H}\left(\rectlow T\left(\Psi,\Phi\right)^{\pm}\left(\e\right)\right)$,
we may assume that the nice sets $\Psi$ and $\Phi$ are compact.

Recall that for every $H\in\left\{ N,A,K\right\} $, $\xi\in H$ and
$\dl>0$, $H_{\left(\xi,\dl\right)}$ denotes the (closed) ball of
radius $\dl$ centered at $\xi$ w.r.t. the metric $d_{H}$ on $H$.
We let $H_{\left(\xi,\dl\right)}^{0}$ denote the corresponding open
ball. By Proposition \ref{prop: Effective Iwasawa decomposition}
there exist positive constants $C_{N}$, $C_{A}$, $C_{K}$ that depend
on $\Psi$ and $\Phi$ alone such that for every $\underline{x}\in\Psi$,
$\phi\in\Phi$, $0<\e\leq\e_{1}$ and $t<0$,
\[
{\cal O}_{\e}\cdot n_{\underline{x}}a_{t}k_{\phi}\cdot{\cal O}_{\e}\subset N_{\left(\underline{x},C_{N}\e\right)}A_{\left(t,C_{A}\e\right)}K_{\left(\phi,C_{K}\e\right)}.
\]
It follows that for every $H\in\left\{ N,A,K\right\} $ and the corresponding
$\Xi\in\left\{ \Psi,\left[-T,0\right],\Phi\right\} $ in $H$,
\begin{equation}
\pi_{H}\left(\rectlow T\left(\Psi,\Phi\right)^{+}\left(\e\right)\right)\subseteq\bigcup_{\xi\in\Xi}H_{\left(\xi,C_{H}\e\right)}\label{eq: supset of the union set}
\end{equation}
and
\begin{equation}
\pi_{H}\left(\rectlow T\left(\Psi,\Phi\right)^{-}\left(\e\right)\right)\supseteq\bigcup_{\xi\in\Xi}H_{\left(\xi,C_{H}\e\right)}^{0}\setminus\bigcup_{\xi\in\del\Xi}H_{\left(\xi,C_{H}\e\right)}^{0}.\label{eq: subset of intersection set}
\end{equation}
Note that the set on the right-hand side of \ref{eq: supset of the union set}
is the union of all $C_{H}\e$-balls that are centered at a point
in $\pi_{H}\left(\rectlow T\left(\Psi,\Phi\right)\right)$, where
the set on the right-hand side of \ref{eq: subset of intersection set}
is the set of points whose (closed) $C_{H}\e$-ball is fully contained
in $\pi_{H}\left(\rectlow T\left(\Psi,\Phi\right)\right)$. \ref{eq: supset of the union set}
is obvious; to see \ref{eq: subset of intersection set}, we note
that
\begin{eqnarray}
g\in\rectlow T\left(\Psi,\Phi\right)^{-}\left(\e\right) & \Longleftrightarrow & g\in u\,\rectlow T\left(\Psi,\Phi\right)v,\,\forall u,v\in\mathcal{O}_{\e}\nonumber \\
 & \Longleftrightarrow & ugv\in\rectlow T\left(\Psi,\Phi\right),\,\forall u,v\in\mathcal{O}_{\e}\nonumber \\
 & \Longleftrightarrow & \pi_{H}\left(ugv\right)\in\pi_{H}\left(\rectlow T\left(\Psi,\Phi\right)\right),\,\forall u,v\in\mathcal{O}_{\e},\label{eq: subset of intersection - explanation}
\end{eqnarray}
since $\mathcal{O}_{\e}=\mathcal{O}_{\e}^{-1}$ and since $\rectlow T\left(\Psi,\Phi\right)$
is a product set. But for every $g$ such that $\pi_{H}\left(g\right)\in\bigcup_{\xi\in\Xi}H_{\left(\xi,C_{H}\e\right)}^{0}\setminus\bigcup_{\xi\in\del\Xi}H_{\left(\xi,C_{H}\e\right)}^{0}$,

\[
\pi_{H}\left(ugv\right)\in H_{\left(\pi_{H}\left(g\right),C_{H}\e\right)}\subset\pi_{H}\left(\rectlow T\left(\Psi,\Phi\right)\right).
\]
Thus every such $g$ is contained in $\rectlow T\left(\Psi,\Phi\right)^{-}\left(\e\right)$
by \ref{eq: subset of intersection - explanation}, and in particular
$\pi_{H}\left(g\right)\in\pi_{H}\left(\rectlow T\left(\Psi,\Phi\right)^{-}\left(\e\right)\right)$
for every $H$.

We begin with the $N$-component. Since $\Psi$ is assumed to be nice,
and since an $\e$-ball in $N$ has $\mu_{N}$-volume which is proportional
to $\e^{\dim N}$, there exists a constant $\a_{1}$ which depends
on $\del\Psi$ and $C_{N}$ such that
\[
\mu_{N}\left(\bigcup_{\underline{x}\in\del\Psi}N_{\left(\underline{x},C_{N}\e\right)}\right)\leq\a_{1}\e^{\dim N}\leq\a_{1}\e.
\]
Thus, by \ref{eq: supset of the union set},
\[
\mu_{N}\left(\pi_{N}\left(\rectlow T\left(\Psi,\Phi\right)^{+}\left(\e\right)\right)\right)\leq\mu_{N}\left(\bigcup_{\underline{x}\in\Psi}N_{\left(\underline{x},C_{N}\e\right)}\right)\leq\mu_{N}\left(\Psi\right)+\mu_{N}\left(\bigcup_{\underline{x}\in\del\Psi}N_{\left(\underline{x},C_{N}\e\right)}\right)\leq\mu_{N}\left(\Psi\right)+\a_{1}\e,
\]
and by \ref{eq: subset of intersection set},
\begin{eqnarray*}
\mu_{N}\left(\pi_{N}\left(\rectlow T\left(\Psi,\Phi\right)^{-}\left(\e\right)\right)\right) & \geq & \mu_{N}\left(\bigcup_{\underline{x}\in\Psi}N_{\left(\underline{x},C_{N}\e\right)}\right)-\mu_{N}\left(\bigcup_{\underline{x}\in\del\Psi}N_{\left(\underline{x},C_{N}\e\right)}^{0}\right)\\
 & = & \mu_{N}\left(\bigcup_{\underline{x}\in\Psi}N_{\left(\underline{x},C_{N}\e\right)}\right)-\mu_{N}\left(\bigcup_{\underline{x}\in\del\Psi}N_{\left(\underline{x},C_{N}\e\right)}\right)\\
 & \geq & \mu_{N}\left(\Psi\right)-\mu_{N}\left(\bigcup_{\underline{x}\in\del\Psi}N_{\left(\underline{x},C_{N}\e\right)}\right)\\
 & \geq & \mu_{N}\left(\Psi\right)-\a_{1}\e.
\end{eqnarray*}
By assuming $\e$ is small enough such that $\a_{1}\e\leq\frac{1}{2}\mu_{N}\left(\Psi\right)$
, the last two inequalities imply
\[
\frac{\mu_{N}\left(\pi_{N}\left(\rectlow T\left(\Psi,\Phi\right)^{+}\left(\e\right)\right)\right)-\mu_{N}\left(\pi_{N}\left(\rectlow T\left(\Psi,\Phi\right)^{-}\left(\e\right)\right)\right)}{\mu_{N}\left(\pi_{N}\left(\rectlow T\left(\Psi,\Phi\right)^{-}\left(\e\right)\right)\right)}\leq\frac{\mu_{N}\left(\Psi\right)+\a_{1}\e-\left(\mu_{N}\left(\Psi\right)-\a_{1}\e\right)}{\frac{1}{2}\mu_{N}\left(\Psi\right)}=\frac{2\a_{1}}{\frac{1}{2}\mu_{N}\left(\Psi\right)}\cdot\e.
\]

The same considerations apply for $\Phi\subseteq K$, since it is
also assumed to be nice, and the $\mu_{K}$-volume of $\e$-balls
in $K$ is proportional to $\e^{\dim K}$. Therefore, there exists
$\a_{2}>0$ that depends on $\del\Phi$ and $C_{K}$ such that
\[
\mu_{K}\left(\bigcup_{k\in\del\Phi}K_{\left(k,C_{K}\e\right)}\right)\leq\a_{2}\e^{\dim K}\leq\a_{2}\e,
\]
and, similarly to the $N$ case, by assuming $\a_{2}\e\leq\frac{1}{2}\mu_{N}\left(\Phi\right)$:
\[
\frac{\mu_{K}\left(\pi_{K}\left(\rectlow T\left(\Psi,\Phi\right)^{+}\left(\e\right)\right)\right)-\mu_{K}\left(\pi_{K}\left(\rectlow T\left(\Psi,\Phi\right)^{-}\left(\e\right)\right)\right)}{\mu_{K}\left(\pi_{K}\left(\rectlow T\left(\Psi,\Phi\right)^{-}\left(\e\right)\right)\right)}\leq\frac{2\a_{2}}{\frac{1}{2}\mu_{K}\left(\Phi\right)}\cdot\e.
\]

Finally, for the $A$-component, it follows from \ref{eq: supset of the union set}
and \ref{eq: subset of intersection set} that
\[
\pi_{A}\left(\rectlow T\left(\Psi,\Phi\right)^{+}\left(\e\right)\right)\subseteq\left[-T-C_{A}\e,0+C_{A}\e\right]
\]
and
\[
\pi_{A}\left(\rectlow T\left(\Psi,\Phi\right)^{-}\left(\e\right)\right)\supseteq\left[-T+C_{A}\e,0-C_{A}\e\right].
\]
Thus,
\[
\mu_{A}\left(\pi_{A}\left(\rectlow T\left(\Psi,\Phi\right)^{+}\left(\e\right)\right)\right)\leq\int_{t=-T-C_{A}\e}^{t=0+C_{A}\e}\frac{dt}{e^{2\rho t}}=\frac{1}{2\rho}\left(e^{2\rho\left(T+C_{A}\e\right)}-e^{-2\rho C_{A}\e}\right)
\]
and
\[
\mu_{A}\left(\pi_{A}\left(\rectlow T\left(\Psi,\Phi\right)^{-}\left(\e\right)\right)\right)\geq\int_{t=-T+C_{A}\e}^{t=0-C_{A}\e}\frac{dt}{e^{2\rho t}}=\frac{1}{2\rho}\left(e^{2\rho\left(T-C_{A}\e\right)}-e^{2\rho C_{A}\e}\right).
\]
As a result,
\begin{eqnarray*}
\frac{\mu_{A}\left(\pi_{A}\left(\rectlow T\left(\Psi,\Phi\right)^{+}\left(\e\right)\right)\right)-\mu_{A}\left(\pi_{A}\left(\rectlow T\left(\Psi,\Phi\right)^{-}\left(\e\right)\right)\right)}{\mu_{A}\left(\pi_{A}\left(\rectlow T\left(\Psi,\Phi\right)^{-}\left(\e\right)\right)\right)} & \leq & \frac{e^{2\rho\left(T+C_{A}\e\right)}-e^{-2\rho C_{A}\e}-\left(e^{2\rho\left(T-C_{A}\e\right)}-e^{2\rho C_{A}\e}\right)}{e^{2\rho\left(T-C_{A}\e\right)}-e^{2\rho C_{A}\e}}\\
 & = & \frac{\left(e^{2\rho T}+1\right)}{e^{2\rho T}}\cdot\frac{\left(e^{2\rho C_{A}\e}-e^{-2\rho C_{A}\e}\right)}{e^{-2\rho C_{A}\e}-e^{-2\rho T}e^{2\rho C_{A}\e}}.
\end{eqnarray*}
For $\e\leq\left(4\rho C_{A}\right)^{-1}$ and $T\geq2\rho^{-1}$
it holds that $e^{2\rho C_{A}\e}-e^{-2\rho C_{A}\e}\leq3\cdot2\rho C_{A}\e$
and $e^{-2\rho C_{A}\e}-e^{-2\rho T}e^{2\rho C_{A}\e}\geq1/2$; therefore,
\[
\frac{\mu_{A}\left(\pi_{A}\left(\rectlow T\left(\Psi,\Phi\right)^{+}\left(\e\right)\right)\right)-\mu_{A}\left(\pi_{A}\left(\rectlow T\left(\Psi,\Phi\right)^{-}\left(\e\right)\right)\right)}{\mu_{A}\left(\pi_{A}\left(\rectlow T\left(\Psi,\Phi\right)^{-}\left(\e\right)\right)\right)}\leq2\cdot\frac{6\rho C_{A}\e}{1/2}=24\rho C_{A}\e.
\]

By choosing $T_{0}=2\rho^{-1}$ and $\e_{0}=\min\left\{ \e_{1},\frac{\mu_{N}\left(\Psi\right)}{2\a_{1}},\frac{\mu_{N}\left(\Phi\right)}{2\a_{2}},\frac{1}{4\rho C_{A}}\right\} $
we conclude that the family $\left\{ \rectlow T\left(\Psi,\Phi\right)\right\} _{T>0}$
is Lipschitz well-rounded, and by Theorem \ref{Thm: GN Counting thm}
(and the discussion in Section \ref{subsec: The GN method}) we are
done.
\end{proof}

\begin{proof}[Proof of Corollary \ref{cor: Equidistribution of N and K}]
Let $\Psi,\Psi',\Phi,\Phi'$ and $\kappa$ as in the statement of the corollary.
By Theorem \ref{thm: Counting in rectangles},
\begin{eqnarray*}
\frac{\#\left(\gam\cap\rectlow T\left(\Psi',\Phi'\right)\right)}{\#\left(\gam\cap\rectlow T\left(\Psi,\Phi\right)\right)} & = & \frac{\mn\left(\Psi'\right)\mu_{K}\left(\Phi'\right)e^{2\rho T}+O\left(T\, e^{2\rho\kappa T}\right)}{\mn\left(\Psi\right)\mu_{K}\left(\Phi\right)e^{2\rho T}+O\left(T\, e^{2\rho\kappa T}\right)}\\
 & = & \frac{\mn\left(\Psi'\right)\mu_{K}\left(\Phi'\right)}{\mn\left(\Psi\right)\mu_{K}\left(\Phi\right)}+O\left(T\left(e^{2\rho T}\right)^{-\left(1-\kappa\right)}\right)
\end{eqnarray*}
(which converges to $\left(\mn\left(\Psi'\right)\mu_{K}\left(\Phi'\right)\right)/\left(\mn\left(\Psi\right)\mu_{K}\left(\Phi\right)\right)$
as $T\ra\infty$, since $\kappa<1$).

Let $\psi$ and $\phi$ be non-negative compactly supported Lipschitz
functions with positive integral, with $\psi$ supported on $N$,
and $\phi$ supported on $K$. Let $R_{T}\left(\psi,\phi\right)$
be the measure on $G$ whose density with respect to Haar measure
on $G$ (written in Iwasawa coordinates as in \ref{eq: Haar measure on G})
is given by the function $D_{T}\left(na_{t}k\right)=\psi\left(n\right)\chi_{[-T,0]}\left(a_{t}\right)\phi\left(k\right)$.
Equivalently, the measure is given by the following formula: for $F\in C_{c}\left(G\right)$,
\[
R_{T}\left(\psi,\phi\right)\left(F\right)=\int_{N}\int_{-T}^{0}\int_{K}F(na_{t}k)\psi(n)\phi(k)d\mu_{N}(n)\frac{dt}{e^{2\rho t}}d\mu_{K}\left(k\right)\,.
\]
The family of measures $R_{T}\left(\psi,\phi\right)$ is Lipschitz
well-rounded, in the following sense. Defining
\[
D_{T}^{+,\epsilon}\left(g\right)=\sup_{u,v\in\mathcal{O}_{\epsilon}}D_{T}\left(ugv\right)\,\,,\,\,D_{T}^{-,\epsilon}\left(g\right)=\inf_{u,v\in\mathcal{O}_{\epsilon}}D_{T}\left(ugv\right)
\]
 we have
\[
\int_{G}D_{T}^{+,\epsilon}\left(g\right)d\mu\left(g\right)\le\left(1+C\epsilon\right)\int_{G}D_{T}^{-,\epsilon}\left(g\right)d\mu\left(g\right)\,.
\]
 Under these assumption, the family $R_{T}\left(\psi,\phi\right)$
satisfies a weighted version of the lattice point counting result
which the sets $R_{T}\left(\Psi,\Phi\right)$ satisfy, namely
\[
\sum_{\gamma\in\Gamma}D_{T}\left(\gamma\right)=\int_{G}D_{T}\left(g\right)d\mu\left(g\right)+O\left(\left(\int_{G}D_{T}\left(g\right)d\mu\left(g\right)\right)^{\kappa\left(\Gamma\right)}\cdot\log\int_{G}D_{T}\left(g\right)d\mu\left(g\right)\right)
\]
 so that in the present case
\[
\sum_{\gamma\in\Gamma}\psi\left(\pi_{N}\left(\gamma\right)\right)\chi_{\left[-T,0\right]}\left(\pi_{A}\left(\gamma\right)\right)\phi\left(\pi_{K}\left(\gamma\right)\right)=
\]
\[
=e^{2\rho T}\int_{N}\psi\left(n\right)d\mu_{N}\left(n\right)\cdot\int_{K}\phi\left(k\right)d\mu_{K}\left(k\right)+O\left(Te^{2\rho T\kappa\left(\Gamma\right)}\right)\,.
\]
The proof of the weighted version of the lattice point problem stated
above under the assumption of Lipschitz well-roundedness is a straightforward
modification of the arguments that appear in \cite{GN1}. The fact
that when $\psi$ and $\phi$ are Lipschitz functions on $N$ and
$K$ the measures $R_{T}\left(\psi,\phi\right)$ defined above are
Lipschitz well-rounded is a straightforward modification of the arguments
in the present paper. Note that it suffices to consider non-negative
Lipschitz functions on $N$ and $K$, and the case of general Lipschitz
functions follows. Finally, the statement of Corollary \ref{cor: Equidistribution of N and K}
part \ref{enu: N e.d.} follows by considering a Lipschitz function
$\psi$ on $N$ supported in $\Psi$, fixing a nice subset $\Phi\subset K$
and letting $\phi$ be its characteristic function, defining $D_{T}$
using $\psi$ and $\phi$, and estimating the ratios as follows
\[
\frac{\sum_{\gamma\in\Gamma}D_{T}\left(\gamma\right)}{\sum_{\gamma\in\Gamma}\chi_{R_{T}\left(\Psi,\Phi\right)}\left(\gamma\right)}=\frac{\int_{N}\psi\left(n\right)d\mu_{N}\left(n\right)}{\mu_{N}\left(\Psi\right)}+O\left(Te^{-\left(1-\kappa\right)T}\right)
\]
The proof of part \ref{enu: K e.d.} is analogous.
\end{proof}

\paragraph*{Acknowledgment.}

We would like to thank to thank Ami Paz for preparing the figures
for this paper. The first author would also like to thank Ami Paz
and Yakov Karasik for helpful discussions.

\bibliographystyle{plain}

\end{document}